\documentclass[reqno, a4paper, 11pt]{amsart}
\usepackage{amsmath, amssymb, amsthm}
\usepackage[hmargin=3cm, vmargin=3cm]{geometry}
\usepackage[usenames]{color}
\usepackage[varg]{txfonts}

\newtheorem{thm}{Theorem}[section]
\newtheorem{cor}[thm]{Corollary}
\newtheorem{lem}[thm]{Lemma}
\newtheorem{prop}[thm]{Proposition}

\theoremstyle{definition}

\theoremstyle{remark}
\newtheorem{rem}[thm]{Remark}
\numberwithin{equation}{section}

\allowdisplaybreaks

\newcommand{\ma}{\mathcal{A}}
\newcommand{\mg}{\mathcal{G}}
\newcommand{\mj}{\mathcal{J}}
\newcommand{\mk}{\mathcal{K}}
\newcommand{\mm}{\mathcal{M}}
\newcommand{\map}{\mathcal{P}}
\newcommand{\mq}{\mathcal{Q}}
\newcommand{\mr}{\mathbb{R}}

\newcommand{\tu}{\tilde{u}}

\newcommand{\ep}{\epsilon}
\newcommand{\pa}{\partial}
\newcommand{\bk}{\boldsymbol\kappa}
\newcommand{\bl}{\boldsymbol\lambda}

\renewcommand{\(}{\left(}
\renewcommand{\)}{\right)}

\begin{document}

\title[Qualitative properties of multi-bubble solutions]{Qualitative properties of multi-bubble solutions for nonlinear elliptic equations involving critical exponents}

\author{Woocheol Choi}
\address[Woocheol Choi]{Department of Mathematical Sciences, Seoul National University, 1 Gwanak-ro, Gwanak-gu, Seoul 151-747, Republic of Korea}
\email{chwc1987@math.snu.ac.kr}

\author{Seunghyeok Kim}
\address[Seunghyeok Kim]{Departamento de Matem\'{a}tica, Pontificia Universidad Cat\'{o}lica de Chile, Avenida Vicu\~{n}a Mackenna 4860, Santiago, Chile}
\email{shkim0401@gmail.com}

\author{Ki-Ahm Lee}
\address[Ki-Ahm Lee]{Department of Mathematical Sciences, Seoul National University, 1 Gwanak-ro, Gwanak-gu, Seoul 151-747, Republic of Korea \& Center for Mathematical Challenges, Korea Institute for Advanced Study, Seoul,130-722, Republic of Korea}
\email{kiahm@math.snu.ac.kr}

\date{\today}

\begin{abstract}
The objective of this paper is to obtain qualitative characteristics of multi-bubble solutions to the Lane-Emden-Fowler equations with slightly subcritical exponents given any dimension $n \ge 3$.
By examining the linearized problem at each $m$-bubble solution, we provide a number of estimates on the first $(n+2)m$-eigenvalues and their corresponding eigenfunctions.
Specifically, we present a new proof of the classical theorem due to Bahri-Li-Rey (1995) \cite{BLR} which states that if $n \ge 4$, then the Morse index of a multi-bubble solution is governed by a certain symmetric matrix whose component consists of a combination of Green's function, the Robin function, and their first and second derivatives.
Our proof also allows us to handle the intricate case $n = 3$.
\end{abstract}

\maketitle

\section{Introduction}
In this paper, we perform a qualitative analysis on the problem
\begin{equation}\label{eq-main-e}\tag{$1.1_{\ep}$}
{\setlength\arraycolsep{2pt}
\left\{\begin{array}{rll}
-\Delta u &= u^{p-\ep} &\text{ in } \Omega,\\
u &> 0 &\text{ in } \Omega,\\
u &= 0 &\text{ on } \pa \Omega,
\end{array}\right.}
\end{equation}
where $\Omega$ is a bounded domain contained in $\mr^n$ $(n \ge 3$), $p = (n+2)/(n-2)$, and $\ep >0$ is a small parameter.
When $\ep > 0$, the compactness of the Sobolev embedding $H^1_0(\Omega) \hookrightarrow L^{p+1-\ep}(\Omega)$ allows one to find its extremal function, hence a positive least energy solution $\bar{u}_{\ep}$ for \eqref{eq-main-e}.
However this does not hold anymore if $\ep = 0$ and in fact existence of solutions strongly depends on topological or geometric properties of the domain in this case (see for instance \cite{BC} and \cite{D}).
If $\ep = 0$ and $\Omega$ is star-shaped, then the supremum of $\bar{u}_{\ep}$ should diverge to $\infty$ as $\ep \to 0$ since an application of the Poho\v{z}aev identity \cite{P} gives nonexistence of a nontrivial solution for \eqref{eq-main-e}.
In the work of Brezis and Peletier \cite{BP}, they deduced the precise asymptotic behavior of $\bar{u}_{\ep}$ when the domain $\Omega$ is the unit ball,
and this result was extended to general domains by Han \cite{H} and Rey \cite{R}, in which they independently proved that $\bar{u}_{\ep}$ blows-up at the unique point that is a critical point of the Robin function of the domain.
Later, Grossi and Pacella \cite{GP} investigated the related eigenvalue problem, obtaining estimates for its first $(n+2)$-eigenvalues, asymptotic behavior of the corresponding eigenvectors and the Morse index of $\bar{u}_{\ep}$.

Let $\{\ep_k\}_{k=1}^{\infty}$ be a sequence of small positive numbers such that $\ep_k \to 0$ as $k \to \infty$ and $\{u_{\ep_k}\}_{k=1}^{\infty}$ a bounded sequence in $H_0^1(\Omega)$ of solutions for \eqref{eq-main-e} with $\ep = \ep_k$, which blow-up at $m \in \mathbb{N}$ points $\{x_{10}, \cdots, x_{m0}\} \subset \overline{\Omega}^m$.
Then by the work of Struwe \cite{S} on the representation of Palais-Smale sequences for \eqref{eq-main-e},
which employed the concentration-compactness principle \cite{Lio}, it can be written as
\setcounter{equation}{1}
\begin{equation}\label{eq-asym}
u_{\ep_k} = \sum_{i=1}^m \alpha_{ik} PU_{\lambda_{ik}\ep_k^{\alpha_0}, x_{ik}} + R_k
\end{equation}
after extracting a subsequence if necessary.
Here $\alpha_0 = 1/(n-2)$, $\alpha_{ik} \to 1$, $\lambda_{ik} \to \lambda_{i0} > 0$ and $x_{ik} \to x_{i0}$ as $k \to \infty$,
$U_{\lambda, x_0}$ is the bubble with the concentration rate $\lambda > 0$ and the center $x_0 \in \mr^n$
\begin{equation}\label{eq-U}
U_{\lambda, x_0} (x) = \beta_n \({\lambda \over \lambda^2 + |x-x_0|^2}\)^{n-2 \over 2} \quad \text{for } x \in \mr^n \quad \text{where } \beta_n = (n(n-2))^{n-2 \over 4},
\end{equation}
the function $PU_{\lambda, x_0}$ is a projected bubble in $H_0^1(\Omega)$, namely, a solution of
\[\Delta PU_{\lambda, x_0} = \Delta U_{\lambda, x_0} \quad \text{in } \Omega, \quad PU_{\lambda, x_0} = 0 \quad \text{on } \pa\Omega\]
and $R_k$ is a remainder term whose $H_0^1(\Omega)$-norm converges to 0 as $k \to \infty$.
Moreover, according to Bahri, Li and Rey \cite{BLR},
if we denote by $G = G(x,y)$ ($x,y \in \Omega$) the Green's function of $-\Delta$ with Dirichlet boundary condition satisfying
\[-\Delta G(\cdot, y) = \delta_y \text{ in } \Omega \quad \text{and} \quad G(\cdot, y) = 0 \text{ on } \pa\Omega,\]
by $H(x,y)$ its regular part, i.e.,
\begin{equation}\label{eq-H}
H(x,y) = {\gamma_n \over |x-y|^{n-2}}-G(x,y) \quad \text{where } \gamma_n = {1 \over (n-2)\left|S^{n-1}\right|},
\end{equation}
and by $\tau$ the Robin function $\tau(x) = H(x,x)$,
then the blow-up rates and the concentration points $(\lambda_{10}, \cdots, \lambda_{m0}, x_{10}, \cdots, x_{m0}) \in (0,\infty)^m \times \Omega^m$ can be characterized as a critical point of the function
\begin{equation}\label{eq-upsilon}
\Upsilon_m(\lambda_1,\cdots, \lambda_m, x_1,\cdots x_m)
= c_1\Bigg( \sum_{i=1}^m \tau(x_i) \lambda_i^{n-2} - \sum_{\substack{i, j=1 \\ i \ne j}}^m G(x_i, x_j) (\lambda_i\lambda_j)^{n-2 \over 2}\Bigg) - c_2\log (\lambda_1 \cdots \lambda_m)
\end{equation}
in general, provided that $n \ge 4$. Here
\begin{equation}\label{eq-upsilon-2}
c_1 = \(\int_{\mr^n} U_{1,0}^p\)^2 \quad \text{and} \quad c_2 = {(n-2)^2 \over 4n}\int_{\mr^n} U_{1,0}^{p+1}.
\end{equation}
Conversely, by applying the Lyapunov-Schmidt reduction method, Musso and Pistoia \cite{MP} proved that if $n \ge 3$ and $(\lambda_{10}, \cdots, \lambda_{m0}, x_{10}, \cdots, x_{m0}) \in (0,\infty)^m \times \Omega^m$ is a $C^1$-stable critical point of $H$ in the sense of Y. Li \cite{L},
then there is a multi-bubbling solution of \eqref{eq-main-e} having the form \eqref{eq-asym} which blows-up at each point $x_{i0}$ with the rate of the concentration $\lambda_{i0}$ ($i = 1, \cdots, m)$.
This extends the existence result also achieved in paper \cite{BLR}, where the authors used the gradient flow of critical points at infinity to get solutions.

Our interest lies on the derivation of certain asymptotic behaviors of solutions $\{u_{\ep}\}_{\ep}$ to \eqref{eq-main-e} satisfying \eqref{eq-asym} when $\ep$ converges to 0.
(Precisely speaking, sequences of parameters $\ep_k$, $\alpha_{ik}$, $\lambda_{ik}$ and $x_{ik}$ in \eqref{eq-asym} should be substituted by $\ep$, $\alpha_{i\ep}$, $\lambda_{i\ep}$ and $x_{i\ep}$, respectively,
such that $\alpha_{i\ep} \to 1$, $\lambda_{i\ep} \to \lambda_{i0}$ and $x_{i\ep} \to x_{i0}$ as $\ep \to 0$.
Hereafter, such a substitution is always assumed.)
It will be done by examining the associated eigenvalue problem
\begin{equation}\label{eq-lin}
{\setlength\arraycolsep{2pt}
\left\{\begin{array}{rll}
-\Delta v &= \mu (p-\ep) u_{\ep}^{p-1-\ep} v &\text{ in } \Omega,\\
v &= 0 &\text{ on } \pa \Omega.
\end{array}\right.}
\end{equation}
We let $\mu_{\ell \ep}$ be the $\ell$-th eigenvalue of \eqref{eq-lin} provided that the sequence of eigenvalues is arranged in nondecreasing order permitting duplication,
and $v_{\ell \ep}$ the corresponding $L^{\infty}(\Omega)$-normalized eigenfunction (namely, $\|v_{\ell \ep}\|_{L^{\infty}(\Omega)} = 1$).

The main aim of this paper is to provide a detailed description on the asymptotic behavior of  $(\mu_{\ell \ep}, v_{\ell \ep})$ for $1 \le \ell \le (n+2)m$.

Firstly, we concentrate on behavior of the first $m$-eigenvalues and eigenvectors.
Given $i,\ \ell \in \mathbb{N}$, $1 \le i \le m$, let $\tilde{v}_{\ell i \ep}$ be a dilation of $v_{\ell \ep}$ defined as
\begin{equation}\label{eq-v-exp}
\tilde{v}_{\ell i \ep}(x) = v_{\ell \ep}\(x_{i \ep} + \lambda_{i \ep}\ep^{\alpha_0} x\) \quad \text{for each } x \in \Omega_{i \ep} := \(\Omega - x_{i \ep}\)/(\lambda_{i \ep}\ep^{\alpha_0}).
\end{equation}

\begin{thm}\label{thm-eigen-zero}
Let $\ep > 0$ be a small parameter,
$\{u_{\ep}\}_{\ep}$ a family of solutions for \eqref{eq-main-e} of the form \eqref{eq-asym},
$\mu_{\ell \ep}$ the $\ell$-th eigenvalue of problem \eqref{eq-lin} for some $1 \le \ell \le m$.
Denote also as $\rho^1_{\ell}$ the $\ell$-th eigenvalue of the symmetric matrix $\ma_1 = \(\ma^1_{ij}\)_{1 \le i,j \le m}$ given by
\begin{equation}\label{eq-mtx-M_0}
\ma^1_{ij} = \begin{cases}
- \(\lambda_{i0}\lambda_{j0}\)^{n-2 \over 2} G\(x_{i0},x_{j0}\) &\text{if } i \ne j,\\
- C_0 + \lambda_{i0}^{n-2} \tau(x_{i0})  &\text{if } i = j,
\end{cases}
\quad \text{where } C_0 = c_2/(c_1(n-2)) > 0.
\end{equation}
Then we have
\begin{equation}\label{eq-eigen-zero-0}
\mu_{\ell \ep} = {n-2 \over n+2} + b_1 \ep + o(\ep) \quad \text{where } b_1 = \({n-2 \over n+2}\)^2 + {(n-2)^3c_1 \over 4n(n+2)c_2}\rho^1_{\ell}
\end{equation}
as $\ep \to 0$.
Moreover, there exists a nonzero column vector
\[\mathbf{c}_{\ell} = \(\lambda_{10}^{n-2 \over 2} c_{\ell 1}, \cdots, \lambda_{m0}^{n-2 \over 2}c_{\ell m}\)^T \in \mr^m\]
such that
for each $i \in \{1, \cdots, m\}$ the function $\tilde{v}_{\ell i \ep}$ converges to $c_{\ell i}U_{1,0}$ weakly in $H^1(\mr^n)$.
This $\mathbf{c}_{\ell}$ becomes an eigenvector corresponding to the eigenvalue $\rho^1_{\ell}$ of $\ma_1$,
and it holds that $\mathbf{c}_{\ell_1}^T \cdot \mathbf{c}_{\ell_2}^T = 0$ for $1 \le \ell_1 \ne \ell_2 \le m$.
\end{thm}

Next, we study the next $mn$-eigenvalues and corresponding eigenvectors.
The first theorem for these eigenpairs concerns with asymptotic behaviors of the eigenvectors.
Let us define a symmetric $m \times m$ matrix $\mm_1 = \(m^1_{ij}\)_{1 \le i, j \le m}$ by
\begin{equation}\label{eq-mtx-M_1}
m^1_{ij} = \begin{cases}
-G\(x_{i0},x_{j0}\) &\text{if } i \ne j,\\
C_0 \lambda_{i0}^{-(n-2)} + \tau\(x_{i0}\) &\text{if } i = j.
\end{cases}
\end{equation}
By Lemma \ref{lem-pos} below, it can be checked that $\mm_1$ is positive definite and in particular invertible.
We denote its inverse by $\(m_1^{ij}\)_{1 \le i, j \le m}$.
\begin{thm}\label{thm-eigen-first}
Assume that $m+1 \le \ell \le (n+1) m$.
Then, for each $i \in \{1,\cdots,m\}$, there exists a vector
$(d_{\ell,i,1}, \cdots, d_{\ell,i,n}) \in \mr^n$, which is nonzero for some $i$, such that
\begin{equation}\label{eq-eigenf-first-1}
\tilde{v}_{\ell i \ep} \to - \sum_{k=1}^n d_{\ell,i,k} \frac{\pa U_{1,0}}{\pa x_k} \quad \text{in } C^1_{\text{loc}}(\mr^n)
\end{equation}
and
\begin{equation}\label{eq-eigenf-first-2}
\begin{aligned}
\ep^{-{n-1 \over n-2}} v_{\ell \ep}(x)
&\to C_1 \left[\sum_{i=1}^m\sum_{j=1}^m\sum_{k=1}^n m_1^{ij} \(-{1 \over 2} \lambda_{j0}^{n-1}d_{\ell,j,k}{\pa \tau \over \pa x_{k0}}(x_{j0})
+ \sum_{l \ne j} \lambda_{l0}^{n-1}d_{\ell,l,k}{\pa G \over \pa y_k} \(x_{j0},x_{l0}\)\) G(x,x_{i0}) \right. \\
&\qquad \left. + \sum_{i=1}^m\sum_{k=1}^n \lambda_{i0}^{n-1}d_{\ell,i,k}{\pa G \over \pa y_k} \(x,x_{i0}\) \right]
\end{aligned}
\end{equation}
in $C^1\(\Omega \setminus \{x_{10}, \cdots, x_{m0}\}\)$ as $\ep \to 0$.
Here $C_1 = \beta_n^p \({n+2 \over n}\) \int_{\mr^n} {|x|^2 \over (1+|x|^2)^{(n+4)/2}} dx > 0$.
\end{thm}
If $\mathbf{d}_{\ell} \in \mr^{mn}$ denotes a nonzero vector defined by
\begin{equation}\label{eq-eigenv-first-2}
\mathbf{d}_{\ell} = \(\lambda_{10}^{n-2 \over 2} d_{\ell,1,1}, \cdots, \lambda_{10}^{n-2 \over 2} d_{\ell,1,n}, \lambda_{20}^{n-2 \over 2} d_{\ell,2,1},
\cdots, \lambda_{(m-1)0}^{n-2 \over 2} d_{\ell,m-1,n}, \lambda_{m0}^{n-2 \over 2} d_{\ell,m,1}, \cdots, \lambda_{m0}^{n-2 \over 2} d_{\ell,m,n}\)^T,
\end{equation}
then we can give a further description on it.
Our next theorem is devoted to this fact as well as a quite precise estimate of the eigenvalues.
Set an $m \times mn$ matrix $\map = (\map_{it})_{1 \le i \le m, 1 \le t \le mn}$ and a symmetric $mn \times mn$ matrix $\mq = (\mq_{st})_{1 \le s, t \le mn}$ as follows.
\begin{equation}\label{eq-mtx-P}
\map_{i, (j-1)n+k} = \begin{cases}
\lambda_{j0}^{n \over 2} \dfrac{\pa G}{\pa y_k}(x_{i0},x_{j0}) = \lambda_{j0}^{n \over 2} \dfrac{\pa G}{\pa x_k}(x_{j0},x_{i0}) &\text{if } i \ne j,\\
- \lambda_{i0}^{n \over 2} \dfrac{1}{2} \dfrac{\pa \tau}{\pa x_k}(x_{i0}) &\text{if } i = j,
\end{cases}
\end{equation}
for $i,\ j \in \{1, \cdots, m\}$ and $k \in \{1, \cdots, n\}$, and
\begin{equation}\label{eq-mtx-Q}
\mq_{(i-1)n+k,(j-1)n+q} = \begin{cases}
\(\lambda_{i0}\lambda_{j0}\)^{n \over 2} \dfrac{\pa^2 G}{\pa x_k \pa y_q}\(x_{i0}, x_{j0}\) &\text{if } i \ne j, \\
-\dfrac{\lambda_{i0}^n}{2} \dfrac{\pa^2\tau}{\pa x_k \pa x_q}\(x_{i0}\)
+ \lambda_{i0}^{n+2 \over 2} \sum\limits_{l \ne i} \lambda_{l0}^{n-2 \over 2} \dfrac{\pa^2 G}{\pa x_k \pa x_q}\(x_{i0}, x_{l0}\) &\text{if } i = j,
\end{cases}
\end{equation}
for $i,\ j \in \{1, \cdots, m\}$ and $k,\ q \in \{1, \cdots, n\}$.

\begin{thm}\label{thm-eigen-first-2}
Let $\ma_2$ be an $mn \times mn$ symmetric matrix
\[\ma_2 = \map^T\mm_1^{-1}\map + \mq.\]
Then as $\ep \to 0$ we have
\begin{equation}\label{eq-eigenv-first}
\mu_{\ell \ep} = 1 - c_0 \rho^2_{\ell} \ep^{n \over n-2} + o\(\ep^{n \over n-2}\)
\end{equation}
for some $c_0 > 0$ (whose value is computed in \eqref{eq-eigen-cha-0}) where $\rho^2_{\ell}$ is the $(\ell-m)$-th eigenvalue of the matrix $\ma_2$.
Furthermore the vector $\mathbf{d}_{\ell} \in \mr^{mn}$ is an eigenvector corresponding to the eigenvalue $\rho^2_{\ell}$ of $\ma_2$, which satisfies $\mathbf{d}_{\ell_1}^T \cdot \mathbf{d}_{\ell_2}^T = 0$ for $m+1 \le \ell_1 \ne \ell_2 \le (n+1)m$.
\end{thm}
\begin{rem}
If the number of blow-up points is $m = 1$, then $\map = 0$ and so the matrix $\ma_2$ is reduced to ${1 \over 2}\lambda_{10}^n D^2\tau(x_{10})$, the Hessian of the Robin function up to a constant multiple,
which is consistent with the result of \cite{GP}.
Note that our Robin function has the opposite sign of that in \cite{GP}, so the sign of the coefficient for $\ep^{n \over n-2}$ in \eqref{eq-eigenv-first} is negative in our case.
See also Remark \ref{rem-est-k_1i}.
\end{rem}

Lastly, the $\ell$-th eigenpair for $(n+1)m+1 \le \ell \le (n+2)m$ can be examined.
Let $\ma_3 = \(\ma^3_{ij}\)_{1 \le i, j \le m}$ be a symmetric matrix whose components are given by
\begin{equation}\label{eq-mtx-A_3}
\ma_{ij}^3 = \begin{cases}
-\(\lambda_{i0}\lambda_{j0}\)^{n-2 \over 2} G\(x_{i0},x_{j0}\) &\text{if } i \ne j,\\
C_0 + \lambda_{i0}^{n-2} \tau(x_{i0}) &\text{if } i = j.
\end{cases}\end{equation}
\begin{thm}\label{thm-eigen-second}
For each $(n+1)m + 1 \le \ell \le (n+2)m$, let $\rho_{\ell}^3$ be the $\(\ell - m(n+1)\)$-th eigenvalue of $\ma_{ij}^3$, which will be shown be positive.
Then there exist a nonzero vector
\begin{equation}\label{eq-hde}
\hat{\mathbf{d}}_{\ell} = \(\lambda_{10}^{n-2 \over 2}d_{\ell,1}, \cdots, \lambda_{m0}^{n-2 \over 2}d_{\ell,m}\)^T \in \mr^m
\end{equation}
and a positive number $c_1$ such that
\[\tilde{v}_{\ell i \ep} \rightharpoonup d_{\ell,i} \(\frac{\pa U_{1,0}}{\pa \lambda}\) \quad \text{weakly in } H^1(\mr^n)\]
and
\[\mu_{\ell \ep} = 1 + c_1 \rho^3_{\ell} \ep + o (\ep) \quad \text{as } \ep \to 0.\]
Furthermore, $\hat{\mathbf{d}}_{\ell}$ is a corresponding eigenvector to $\rho^3_{\ell}$,
and it holds that $\hat{\mathbf{d}}_{\ell_1}^T \cdot \hat{\mathbf{d}}_{\ell_2}^T = 0$ for $(n+1)(m+1) \le \ell_1 \ne \ell_2 \le (n+2)m$.
\end{thm}

\noindent As a result, we obtain the following corollary.
\begin{cor}
Let $\text{ind}(u_{\ep})$ and $\text{ind}_0(u_{\ep})$ be the morse index and the augmented Morse index of the solution $u_{\ep}$ to \eqref{eq-main-e}, respectively.
Also for the matrix $\ma_2$ in Theorem \ref{thm-eigen-first-2},
$\text{ind}(-\ma_2)$ and $\text{ind}_0(-\ma_2)$ are similarly understood. Then
\[m \le m + \text{ind}(-\ma_2) \le \text{ind}(u_{\ep}) \le \text{ind}_0(u_{\ep}) \le m + \text{ind}_0(-\ma_2) \le (n+1)m\]
for sufficiently small $\ep > 0$.
Therefore if $\ma_2$ is nondegenerate, then so is $u_{\ep}$ and
\[\text{ind}(u_{\ep}) = m + \text{ind}(-\ma_2) \in [m, (n+1)m].\]
\end{cor}
\begin{rem}
By the discussion before, our results hold for solutions found by Musso and Pistoia in \cite{MP}.
Moreover, if $\ep_k \to 0$ as $k \to \infty$, any $H_0^1(\Omega)$-bounded sequence $\{u_{\ep_k}\}_{k=1}^{\infty}$ of solutions for \eqref{eq-main-e} with $\ep = \ep_k$
has a subsequence to which our work can be applied.
\end{rem}

\noindent This extends the work of Bahri-Li-Rey \cite{BLR} where the validity of the above corollary was obtained for $n \ge 4$.
Besides Theorems \ref{thm-eigen-zero}, \ref{thm-eigen-first}, \ref{thm-eigen-first-2} and \ref{thm-eigen-second} provide sharp asymptotic behaviors of the eigenpairs $(\mu_{\ell\ep}, v_{\ell\ep})$ as $\ep \to 0$ which were not dealt with in \cite{BLR}.
In this article we compute each component of the matrix $\ma_2$ explicitly, which turns out to be complicated.
Instead doing in this way, the authors of \cite{BLR} gave an alternative neat description.

Our proof is based on the work of Grossi and Pacella \cite{GP} which studied qualitative behaviors of single blow-up solutions of \eqref{eq-main-e},
but requires a further inspection on the interaction between different bubbles here.
In particular we have to control the decay of solutions $u_{\ep}$ and eigenfunctions $v_{\ell \ep}$ near each blow-up point in a careful way.
In order to get the sharp decay of $u_{\ep}$, we will utilize the method of moving spheres which has been used on equations from conformal geometry and related areas. (See for example \cite{CL, CC, LZh, Pa}.)
Furthermore we shall make use of the Moser-Harnack type estimate and an iterative comparison argument to find an almost sharp decay of $v_{\ell \ep}$.

\medskip
Before starting the proof of our main theorems, we would like to mention about related results obtained for the Gelfand problem
\[{\setlength\arraycolsep{2pt}
\left\{\begin{array}{rll}
-\Delta u &= \lambda e^u &\text{ in } \Omega,\\
u &= 0 &\text{ on } \pa \Omega,
\end{array}\right.}\]
where $\Omega$ is a bounded smooth domain in $\mr^2$ and $\lambda > 0$ is a small parameter.
In \cite{GG}, given $u_{\lambda}$ an one-bubble solution satisfying $\lambda \int_{\Omega}e^{u_{\lambda}} \to 8\pi$ as $\lambda \to 0$,
Gladiali and Grossi obtained the asymptotic behavior of the eigenvalues $\mu$ for the problem
\[{\setlength\arraycolsep{2pt}
\left\{\begin{array}{rll}
-\Delta v &= \lambda \mu e^{u_{\lambda}}v &\text{ in } \Omega,\\
v &= 0 &\text{ on } \pa \Omega,
\end{array}\right.}\]
and the Morse index of $u_{\lambda}$ as a by-product.
Recently, such a type of results has been generalized to solutions with multiple blow-up points in \cite{GGOS},
and further qualitative properties of the first $m$ eigenfunctions has been described in \cite{GGO} when $m$ designates the number of blow-up points.

Also, we believe that there should be analogue to our main results for solutions of the Brezis-Nirenberg problem \cite{BN}
\[{\setlength\arraycolsep{2pt}
\left\{\begin{array}{rll}
-\Delta u &= u^p + \ep u &\text{ in } \Omega,\\
u &> 0 &\text{ in } \Omega,\\
u &= 0 &\text{ on } \pa \Omega,
\end{array}\right.}\]
where $\Omega$ is a bounded smooth domain of $\mr^n$ ($n \ge 5$), if asymptotic forms of the solutions are written as
\[u_{\ep} = \sum_{i=1}^m PU_{\lambda_{i \ep}\ep^{{1}/{(n-4)}}, x_{i \ep}} + R_{\ep}\]
for $\lambda_{i\ep} \to \lambda_{i0} > 0$, $x_{i \ep} \to x_{i0}$ and $R_\ep \to 0$ in $H^1_0(\Omega)$ as $\ep \to 0$.
This type of solutions was obtained by Musso and Pistoia \cite{MP}, while Takahashi \cite{T} analyzed the linear problem of one-bubble solutions.

\medskip
The structure of this paper can be described in the following way.
In Section \ref{sec-pre}, we gather all preliminary results necessary to deduce our main theorems.
This section in particular includes estimates of the decay of the solutions $u_{\ep}$ or the eigenfunctions $v_{\ell \ep}$ outside of the concentration points $\{x_{10}, \cdots, x_{m0}\}$.
In Section \ref{sec-3}, we prove Theorem \ref{thm-eigen-zero} which deals with the first $m$-eigenvalues and eigenfunctions of problem \eqref{eq-lin}.
A priori bounds for the first $(n+1)m$-eigenvalues and the limit behavior \eqref{eq-eigenf-first-1} of expanded eigenfunction $\tilde{v}_{\ell i \ep}$ are found in Section \ref{sec-4}.
Based on these results, we compute an asymptotic expansion \eqref{eq-eigenf-first-2} of the $\ell$-th eigenvectors ($\ell = m+1, \cdots, (n+1)m$)
and that of its corresponding eigenvalues \eqref{eq-eigenv-first} in Sections \ref{sec-5} and \ref{sec-6} respectively.
The description of the vector $\mathbf{d}_{\ell}$ is also obtained as a byproduct during the derivation of \eqref{eq-eigenv-first}.
Section \ref{sec-7} is devoted to study the next $m$-eigenpairs, i.e., the $\ell$-th eigenvalues and eigenfunctions ($\ell = (n+1)m+1, \cdots, (n+2)m$).
Finally, we present the proof of Proposition \ref{prop-LZ} in Appendix \ref{sec-mov-sphere}, which is conducted with the moving sphere method.

\bigskip
\noindent \textbf{Notations.}

\medskip \noindent - Big-O notation and little-o notation are used to describe the limit behavior of a certain quantity as $\ep \to 0$.

\medskip \noindent - $B^n(x,r)$ is the $n$-dimensional open ball whose center is located at $x$ and radius is $r$.
Also, $S^{n-1}$ is the $(n-1)$-dimensional unit sphere and $\left|S^{n-1}\right|$ is its surface area.

\medskip \noindent - $C > 0$ is a generic constant which may vary from line to line, while numbers with subscripts such as $c_0$ or $C_1$ have positive fixed values.

\medskip \noindent - $D^{1,2}(\mr^n) = \left\{u \in L^{\frac{2n}{n-2}}(\mr^n) \mid \int_{\mr^n} |\nabla u|^2 < \infty\right\}$.

\medskip \noindent - For any number $c \in \mr$, $c = c_+ - c_-$ where $c_+, c_- \ge 0$ are the positive or negative part of $c$, respectively.

\medskip \noindent - For any vector $\mathbf{v}$, its transpose is denoted as $\mathbf{v}^T$.

\section{Preliminaries} \label{sec-pre}
In this section, we collect some results necessary for our analysis.
For the rest of the paper, we write $x_1, \cdots, x_m$ to denote the concentration points, dropping out the subscript 0.
The same omission also applies to the concentrate rates $\lambda_1, \cdots, \lambda_m$.

\begin{lem}\label{lem-pos}
If we set a matrix $\mm_2 = \(m^2_{ij}\)_{1\le i, j \le m}$ by
\begin{equation}\label{eq-mtx-M_2}
m^2_{ij} = \begin{cases}
-G(x_i,x_j) &\text{if } i \ne j,\\
\tau(x_i) &\text{if } i = j,
\end{cases}\end{equation}
then it is a non-negative definite matrix.
\end{lem}
\begin{proof}
See Appendix A of Bahri, Li and Rey \cite{BLR}.
\end{proof}

Fix any $i \in \{1, \cdots, m\}$ and decompose $u_{\ep}$ in the following way.
\begin{equation}\label{eq-u-exp-2}
u_{\ep} = U_{\lambda_{i \ep}\ep^{\alpha_0}, x_{i \ep}} + \(PU_{\lambda_{i \ep}\ep^{\alpha_0}, x_{i \ep}} - U_{\lambda_{i \ep}\ep^{\alpha_0}, x_{i \ep}}\) + (\alpha_{i \ep} - 1) PU_{\lambda_{i \ep}\ep^{\alpha_0}, x_{i \ep}} +
\sum_{j \ne i} \alpha_{j \ep} PU_{\lambda_{i \ep}\ep^{\alpha_0}, x_{i \ep}} + R_{\ep}.
\end{equation}
Then we rescale it to define
\begin{equation}\label{eq-u-exp}
\tu_{i \ep}(x) = (\lambda_{i \ep}\ep^{\alpha_0})^{\sigma_{\ep}} u_{\ep}\(x_{i \ep} + \lambda_{i \ep}\ep^{\alpha_0} x\) \quad \text{where } \sigma_{\ep} = {2 \over p-1-\ep} = {n-2 \over 2 - (n-2)\ep/2}.
\end{equation}
It immediately follows that $\{\tu_{i \ep}\}_{\ep}$ is a family of positive $C^2$-functions defined in $B^n\(0, \ep^{-\alpha_0}r_0\)$ for some $r_0 > 0$ small enough (determined in the next lemma),
which are solutions of $-\Delta u = u^{p-\ep}$.
Moreover it has the following property.
\begin{lem}\label{lem-u-dil}
The sequence $\{\tu_{i \ep}\}_{\ep}$ satisfies $\|\tu_{i \ep}\|_{L^{\infty}\(B^n\(0, \ep^{-\alpha_0}r_0\)\)} \le c$ for some small $r_0 > 0$ and converges to $U_{1,0}$ weakly in $H^1(\mr^n)$ as $\ep \to 0$.
\end{lem}
\begin{proof}
For fixed $i$, let us denote $\tilde{f}(x) = (\lambda_{i \ep}\ep^{\alpha_0})^{\sigma_{\ep}} f\(x_{i \ep} + \lambda_{i \ep}\ep^{\alpha_0} x\)$ for $x \in \Omega_{i \ep} = \(\Omega - x_{i \ep}\)/(\lambda_{i \ep}\ep^{\alpha_0})$.
Set also $U_j = U_{\lambda_{j \ep}\ep^{\alpha_0}, x_{j \ep}}$ for all $j \in \{1, \cdots, m\}$.
Then $\|\tilde{f}\|_{H^1(\Omega_{i\ep})} = (1+o(1))\|f\|_{H^1(\Omega)}$ and
\begin{equation}\label{eq-u-dil}
\tu_{i \ep} - U_{1,0} = \sum_{j \ne i} \alpha_{j \ep} \widetilde{PU}_j + \(\widetilde{PU}_i - \widetilde{U}_i\) + (\alpha_{i \ep} - 1) \widetilde{PU}_i + \widetilde{R}_{\ep} \quad \text{in } \Omega_{i\ep}
\end{equation}
by \eqref{eq-u-exp-2}.
Observe that $0 \le PU_i \le U_i$ in $\Omega$ and
\[PU_{\lambda, x_0}(x) = U_{\lambda, x_0}(x) - C_2\lambda^{n-2 \over 2} H(x, x_0) + o\(\lambda^{n-2 \over 2}\) \quad \text{in } C^0\(\overline{\Omega}\), \quad C_2 := \int_{\mr^n} U_{1,0}^p > 0\]
holds for any small $\lambda > 0$ and $x_0 \in \Omega$ away from the boundary. Thus
\[\|PU_i - U_i\|_{H^1(\Omega)}^2 = -\int_{\Omega} U_i^p(PU_i-U_i) - \int_{\Omega} U_i^{p+1} + \int_{\Omega} |\nabla U_i|^2 = o(1)\]
and
\[\|PU_i\|_{H^1(\Omega)}^2 = \int_{\Omega}U_i^p PU_i \le \int_{\mr^n} U_{1,0}^{p+1}\]
so that the last three terms in the right-hand side of \eqref{eq-u-dil} go to 0 strongly in $H^1_0(\Omega_{i \ep}) \subset H^1(\mr^n)$.
On the other hand, we have
\[\left| \int_{\text{supp}(\varphi)} \nabla \widetilde{PU}_j \cdot \nabla \varphi \right| \le \|\varphi\|_{L^{\infty}(\Omega)} \int_{\text{supp}(\varphi)} \widetilde{U}_j^{p-\ep} \to 0\]
as $\ep \to 0$ for any test function $\varphi \in C^{\infty}_c(\mr^n)$.  Therefore $\tu_{i\ep} \rightharpoonup U_{1,0}$ weakly in $H^1(\mr^n)$.

We now attempt to attain a priori $L^{\infty}$-estimate for $\{\tu_{i \ep}\}_{\ep}$.
Firstly we fix a sufficiently small $r_0$.
In fact, $r_0 = {1 \over 2}\min\left\{|x_i-x_j|: i, j = 1, \cdots, m \text{ and } i \ne j\right\} > 0$ would suffice.
Then for any number $\eta > 0$, one can find $r > 0$ small such that $\left\|\tu_{i \ep}^{p-1-\ep}\right\|_{L^{n \over 2}(B^n(x,r))} \le \eta$ is valid for any $|x| \le \ep^{-\alpha_0}r_0$ provided $\ep > 0$ sufficiently small.
Hence the Moser iteration technique applies as in \cite[Lemma 6]{H}, deducing
\[\|\tu_{i \ep}\|_{L^{(p+1){n \over n-2}}(B^n(x,r/2))} \le {C \over r} \|\tu_{i \ep}\|_{L^{p+1}(B^n(x,r))} \le {C \over r} \|\tu_{i \ep}\|_{H^1(\Omega_{i \ep})}\]
where the rightmost value is uniformly bounded in $\ep > 0$.
Also it is notable that $C > 0$ is independent of $x$, $r$ or $\tu_{i\ep}$.
As a result, we observe from the elliptic regularity \cite[Lemma 7]{H} that
\[|u(x)| \le \|u\|_{L^{\infty}(B^n(x,r/4))} \le C\|u\|_{L^{p+1}(B^n(x,r/2))}\]
where $C > 0$ depends only on $r$ and the supreme of $\left\{\|\tu_{i \ep}\|_{L^{(p+1){n \over n-2}}(B^n(x,r/2))}\right\}_{\ep}$.
This completes the proof.
\end{proof}
\noindent This lemma will be used in a crucial way to deduce a local uniform estimate near each blow-up point $x_1, \cdots, x_m$ of $u_{\ep}$.
\begin{prop}\label{prop-LZ}
There exist numbers $C>0$ and small $\delta_0 \in (0, r_0)$ independent of $\ep > 0$ such that
\begin{equation}\label{eq-LZ}
\tu_{i \ep} (x) \le CU_{1,0}(x) \quad \text{for all } x \in B^n\(0, \ep^{-\alpha_0}\delta_0\)
\end{equation}
for all sufficiently small $\ep > 0$.
\end{prop}
\noindent A closely related result to Proposition \ref{prop-LZ} appeared in \cite{LZ} as an intermediate step to deduce the compactness property of the Yamabe equation,
the problem proposed by Schoen who also gave the positive answer for conformally flat manifolds (see \cite{Sch}).
Even though the proof of this proposition, based on the moving sphere method, can be achieved by adapting the argument presented in \cite{LZ} with a minor modification,
we provide it in Appendix \ref{sec-mov-sphere} to promote clear understanding of the reader.

\bigskip

From the next lemma to Lemma \ref{lem-u-asym}, we study the behavior of solutions $u_{\ep}$ of \eqref{eq-main-e} outside the blow-up points $\{x_1, \cdots, x_m\}$.
For the sake of notational convenience, we set
\begin{equation}\label{eq-A_r}
A_r = \Omega \setminus \cup_{i=1}^m B^n(x_i, r) \quad \text{for any } r > 0.
\end{equation}
\begin{lem}\label{lem-u-zero}
Suppose that $\{u_{\ep}\}_{\ep}$ is a family of solutions for \eqref{eq-main-e} satisfying the asymptotic behavior \eqref{eq-asym}.
Then for any $r > 0$, we have $u_{\ep}(x) = o(1)$ uniformly for $x \in A_r$.
\end{lem}
\begin{proof}
Let $a_{\ep} = u_{\ep}^{p-1-\ep}$ so that $-\Delta u_{\ep} = a_{\ep}u_{\ep}$ in $\Omega$.
Then we see from \eqref{eq-asym} that
\begin{align*}
\|a_{\ep}\|_{L^{n \over 2}(A_{r/4})}
& \le C \(\sum_{i=1}^m \left\|PU_{\lambda_{i \ep}\ep^{\alpha_0},x_{i \ep}}^{p-1-\ep}\right\|_{L^{n \over 2}(A_{r/4})} + \|R_{\ep}\|_{L^{p+1-\ep {n \over 2}}(A_{r/4})}^{p-1-\ep}\)\\
& \le C \(\sum_{i=1}^m \left\|U_{\lambda_{i \ep}\ep^{\alpha_0},x_{i \ep}}^{p-1-\ep}\right\|_{L^{n \over 2}\(\mr^n \setminus B^n(x_i,{r/4})\)} + \|R_{\ep}\|_{H^1(\Omega)}^{p-1-\ep}\) = o(1).
\end{align*}
Therefore we can proceed the Morse iteration argument as in the proof of \cite[Lemma 6]{H} to get $\|a_{\ep}\|_{L^q(A_{r/2})} = o(1)$ for some $q > n/2$, and then the standard elliptic regularity result (see \cite[Lemma 7]{H}) implies $\|u_{\ep}\|_{L^{\infty}(A_r)} = o(1)$.
\end{proof}

\noindent We can improve this result by combining the kernel expression of $u_{\ep}$ and Proposition \ref{prop-LZ}.
\begin{lem}\label{lem-u-bound}
Fix $r > 0$ small. Then there holds
\begin{equation}\label{eq-u-bound}
u_{\ep}(x) = O\(\sqrt{\ep}\)
\end{equation}
uniformly for $x \in A_r$.
\end{lem}
\begin{proof}
Without any loss of generality, we may assume that $r \in (0,\delta_0)$ where $\delta_0 > 0$ is the number picked up in Proposition \ref{prop-LZ} so that \eqref{eq-LZ} holds.
Thus if we fix $i \in \{1, \cdots, m\}$, then we have the bound
\[u_{\ep}(x) = (\lambda_{i \ep}\ep^{\alpha_0})^{-\sigma_{\ep}} \tu_{i \ep}\(\(\lambda_{i \ep}\ep^{\alpha_0}\)^{-1}(x-x_{i \ep})\)
\le C U_{\lambda_{i \ep}\ep^{\alpha_0},x_{i \ep}}(x) \le C \ep^{\({n-2 \over 2}\)\alpha_0}\]
valid for each $x$ such that $r/2 \le |x - x_i| \le r$.
It says that $u_{\ep}(x) \le C\sqrt{\ep}$ for all $x \in A_r \setminus A_{r/2}$.

By Green's representation formula, one may write
\begin{equation}\label{eq-u-decom}
u_{\ep}(x) = \int_{A_{r/2}} G(x,y) u_{\ep}^{p-\ep}(y) dy + \sum_{i=1}^m \int_{B^n(x_i,r/2)} G(x,y) u_{\ep}^{p-\ep} (y)dy.
\end{equation}
Let us estimate each of the term in the right-hand side. If we set $b_{\ep} = \max\{u_{\ep}(x): x \in A_r\}$, then we find
\begin{equation}\label{eq-bound-1}
\int_{A_{r/2}} G(x,y) u_{\ep}^{p-\ep}(y) dy \le
C \int_{A_{r/2}} G(x,y) \(b_{\ep}^{p-\ep} + \sqrt{\ep}^{p-\ep}\) dy
\le C \(b_{\ep}^{p-\ep} + C \sqrt{\ep}^{p-\ep}\)
\end{equation}
for any $x \in A_r$. Besides, \eqref{eq-LZ} gives us that
\begin{equation}\label{eq-bound-2}
\begin{aligned}
\int_{B^n(x_i,r)} G(x,y) u_{\ep}^{p-\ep} (y)dy &\le C(r) \int_{B^n(x_i,r)}u_{\ep}^{p-\ep} (y)dy
\le C \cdot C(r) \int_{B^n(x_i,r)}U_{\lambda_{i \ep}\ep^{\alpha_0},x_{i \ep}}^{p-\ep}(y)dy \\
&\le C \cdot C(r) \ep^{\({n-2 \over 2}\)\alpha_0} = C \cdot C(r) \sqrt{\ep}
\end{aligned}
\end{equation}
for each $i$ and $x \in A_r$, where $C(r) = \max\{G(x,y): x, y \in \Omega, |x-y| \ge r/2\}$.
Hence, by combining \eqref{eq-bound-1} and \eqref{eq-bound-2}, we get
\[b_{\ep} \le C \(b_{\ep}^{p-\ep} + \sqrt{\ep}\).\]
Since it is guaranteed by Lemma \ref{lem-u-zero} that $b_{\ep} = o(1)$, this shows that $b_{\ep} \le C\sqrt{\ep}$.
The lemma is proved.
\end{proof}
\noindent The following result will be used to obtain the asymptotic formulas of the eigenvalues.
\begin{lem}\label{lem-u-asym}
Suppose that $u_{\ep}$ satisfies equation \eqref{eq-main-e} and the asymptotic behavior \eqref{eq-asym}.
Then we have
\begin{equation}\label{eq-u-ep-asym}
\ep^{-\frac{1}{2}} \cdot u_{\ep}(x) = C_2 \sum_{i=1}^m \lambda_i^{\frac{n-2}{2}} G(x,x_i) + o(1)
\end{equation}
in $C^2(\Omega \setminus\{x_1,\cdots,x_m\})$.
Here $C_2 = \int_{\mr^n} U_{1,0}^p > 0$.
\end{lem}
\begin{proof}
Take any $r > 0$ small for which Lemma \ref{lem-u-bound} holds
and decompose $u_{\ep}(x)$ as in \eqref{eq-u-decom} for $x \in A_r$.
Then we have
\begin{equation}\label{eq-2-2}
\left|{\ep}^{-{1 \over 2}} \int_{A_{r/2}} G(x,y) u_{\ep}^{p-\ep}(y) dy \right| \le C {\ep}^{{p-1-\ep \over 2}}\(\int_{\Omega} G(x,y) dy\) = o(1).
\end{equation}
Also, if we write
\[\int_{B^n(x_i,r/2)} G(x,y) u_{\ep}^{p-\ep} (y)dy = G(x,x_i)  \int_{B^n(x_i,r/2)} u_{\ep}^{p-\ep} (y)dy + \int_{B^n(x_i,r/2)} (G(x,y)-G(x,x_i)) u_{\ep}^{p-\ep} (y)dy\]
for $i \in \{1, \cdots, m\}$, it follows from Lemma \ref{lem-u-dil} and the dominated convergence theorem that
\begin{equation}\label{eq-2-3}
\ep^{-{1 \over 2}} \int_{B^n(x_i, r/2)} u_{\ep}^{p-\ep}(y) dy \to \lambda_i^{\frac{n-2}{2}} \int_{\mr^n} U_{1,0}^p (y) dy = \lambda_i^{\frac{n-2}{2}} C_2
\end{equation}
and from the mean value theorem that
\begin{equation}\label{eq-2-4}
\left|\ep^{-{1 \over 2}} \int_{B^n(x_i,r/2)} (G(x,y)-G(x,x_i)) u_{\ep}^{p-\ep}(y) dy\right| \le \ep^{-{1 \over 2}} \int_{B^n(x_i,r/2)} |G(x,y)-G(x,x_i)| u_{\ep}^{p-\ep}(y) dy = O(r).
\end{equation}
Therefore, combining \eqref{eq-u-decom}, \eqref{eq-2-2}, \eqref{eq-2-3} and \eqref{eq-2-4}, we confirm that
\[C_2 \sum_{i=1}^m \lambda_i^{\frac{n-2}{2}} G(x,x_i) - C r \le \liminf_{\ep \to 0} {\ep}^{-{1 \over 2}} u_{\ep}(x) \le \limsup_{\ep \to 0} {\ep}^{-{1 \over 2}} u_{\ep}(x) \le C_2 \sum_{i=1}^m \lambda_i^{\frac{n-2}{2}} G(x,x_i) + Cr.\]
Since $r>0$ is arbitrary, \eqref{eq-u-ep-asym} holds in $C^0 (\Omega \setminus\{x_1,\cdots,x_m\})$.
Also, the $C^2$-convergence comes from the elliptic regularity. 
This proves the lemma.
\end{proof}

\medskip
In Lemma \ref{lem-con-ext} and Lemma \ref{lem-w-bound}, we conduct a decay estimate for solutions of the eigenvalue problem \eqref{eq-lin}.
\begin{lem}\label{lem-con-ext}
For a fixed $\ell \in \mathbb{N}$, let $\{\mu_{\ell \ep}\}_{\ep}$ be the family of $\ell$-th eigenvalues for problem \eqref{eq-lin},
and $v_{\ell \ep}$ an $L^{\infty}(\Omega)$-normalized eigenfunction corresponding to $\mu_{\ell \ep}$.
Then for any $r>0$ the function $v_{\ell \ep}$ converges to zero uniformly in $A_r$ as $\ep \to 0$.
\end{lem}
\begin{proof}
For $x \in A_r$, we write
\begin{equation}\label{eq-5-1}
\frac{v_{\ell \ep}(x)}{{\mu_{\ell \ep}}(p-\ep)} = \int_{A_{r/2}}G(x,y) u_{\ep}^{p-1-\ep} v_{\ell \ep}(y) dy + \sum_{i=1}^m \int_{B^n(x_i, r/2)} G(x,y) u_{\ep}^{p-1-\ep} v_{\ell \ep}(y) dy.
\end{equation}
From Lemma \ref{lem-u-bound}, we have
\begin{equation}\label{eq-5-2}
\left| \int_{A_{r/2}} G(x,y) \(u_{\ep}^{p-1-\ep} v_{\ell \ep}\)(y) dy\right| \le C \cdot \ep^{p-1-\ep \over 2} \(\int_{\Omega} G(x,y) dy\) = O\(\ep^{\frac{2}{n-2}}\).
\end{equation}
Also, we utilize \eqref{eq-LZ} to obtain that
\begin{equation}\label{eq-5-3}
\begin{aligned}
\left| \int_{B^n(x_i, r/2)} G(x,y) \(u_{\ep}^{p-1-\ep} v_{\ell \ep}\)(y) dy\right|
&\le C(r) \int_{B^n(x_i, r/2)} u_{\ep}^{p-1-\ep} (y) dy\\
&\le C \cdot C(r) \int_{B^n(0, r)}U^{p-1-\ep}_{\lambda_{i \ep}\ep^{\alpha_0},0}(y)dy\\
&= O \({\ep}^{\frac{2}{n-2}}\) \text{ if } n \ge 5,\ O(\ep \log{\ep}) \text{ if } n = 4,\ O(\ep) \text{ if } n = 3
\end{aligned}
\end{equation}
for any $1 \le i \le m$ where the definition of $C(r)$ can be found in the sentence after \eqref{eq-bound-2}.
Putting estimates \eqref{eq-5-2} and \eqref{eq-5-3} into \eqref{eq-5-1} validates that $v_{\ell \ep} = o(1)$ uniformly in $A_r$.
\end{proof}

\begin{lem}\label{lem-w-bound}
Assume that $0 \in \Omega$, fix $\ell \in \mathbb{N}$ and set
\[\tilde{v}_{\ell \ep} = v_{\ell \ep}\(\ep^{\alpha_0}x\) \quad \text{and} \quad d_{\ep}(x) = \text{dist}\(x, \left\{\ep^{-\alpha_0}x_{1\ep}, \cdots, \ep^{-\alpha_0}x_{m\ep} \right\}\) \quad \text{for } x \in \Omega_{\ep} := \ep^{-\alpha_0}\Omega.\]
Then for any $\zeta > 0$ small, we can pick a constant $C = C(\zeta) > 0$ independent of $\ep > 0$ such that
\begin{equation}\label{eq-w-bound-0}
|\tilde{v}_{\ell \ep}(x)| \le {C \over 1+d_{\ep}(x)^{n-2-\zeta}} \quad \text{for all } x \in \Omega_{\ep}.
\end{equation}
In particular, if $i \in \{1, \cdots, m\}$ are given and $\{\tilde{v}_{\ell i \ep}\}_{\ep}$ is a family of dilated eigenfunctions for \eqref{eq-main-e} defined as in \eqref{eq-v-exp}, then
\begin{equation}\label{eq-w-bound-3}
|\tilde{v}_{\ell i \ep}(x)| \le \frac{C}{1+|x|^{n-2-\zeta}} \quad \text{for all } |x| \le \ep^{-\alpha_0}r
\end{equation}
and $v_{\ep} = O(\ep)$ in $A_r$ for some $r > 0$ small.
\end{lem}
\begin{proof}
One can derive the decay estimate \eqref{eq-w-bound-0} by adapting the proof of Lemmas A.5, B.3 and Proposition B.1 of Cao, Peng and Yan \cite{CPY} in a suitable way, in which the authors investigated the $p$-Laplacian version of the Brezis-Nirenberg problem.
We sketch it for the completeness.
Let $\tu_{\ep} = u_{\ep}\(\ep^{\alpha_0}\cdot\)$ and $\tilde{x}_{i\ep} = \ep^{-\alpha_0}x_{i\ep}$.

Notice that $\tilde{v}_{\ell \ep}$ solves
\[-\Delta \tilde{v}_{\ell \ep} = a_{\ell \ep}\tilde{v}_{\ell \ep} \quad \text{in } \Omega_{\ep} \quad \text{where} \quad a_{\ell \ep} = \mu_{\ell \ep}(p-\ep) \ep^{2\alpha_0} \tu_{\ep}^{p-1-\ep} \ge 0.\]
From Proposition \ref{prop-LZ} and Lemma \ref{lem-u-bound}, we realize that $a_{\ell \ep} \le C |x|^{-4+(n-2)\ep}$ holds in each annulus $B^n\(\tilde{x}_{i\ep},\delta_0 \ep^{-\alpha_0}\) \setminus B^n(\tilde{x}_{i\ep},R)$
provided $i \in \{1, \cdots, m\}$ and $R > 1$ large, and $a_{\ell \ep} \le C \ep^{4\alpha_0}$ in $\Omega_{\ep} \setminus \cup_{i=1}^m B^n\(\tilde{x}_{i\ep},\delta_0 \ep^{-\alpha_0}\)$.
Hence, given any $\eta > 0$, there exists a large $R(\eta) > 1$ such that
\begin{equation}\label{eq-w-bound-1}
\int_{\widetilde{A}_{R(\eta)}} |a_{\ell \ep}|^{n \over 2} dx < \eta \quad \text{where} \quad  \widetilde{A}_R := \Omega_{\ep} \setminus \bigcup_{i=1}^m B^n(\tilde{x}_{i\ep}, R).
\end{equation}
Suppose that $\zeta > 0$ is selected to be small enough.
Then one can apply the Moser iteration technique to get a small number $\eta > 0$ and large $q > p+1$ such that if \eqref{eq-w-bound-1} holds, there is a constant $C > 0$ independent of $R$, $\eta$ or $\tilde{v}_{\ell \ep}$ satisfying
\[\|\tilde{v}_{\ell \ep}\|_{L^q\(\widetilde{A}_R\)} \le {C \over (R-2R(\eta))^{{n-2 \over 2}-\zeta}} \cdot \|\tilde{v}_{\ell \ep}\|_{L^{p+1}\(\widetilde{A}_{2R(\eta)}\)}\]
for any $R > 2R(\eta)$.
On the other hand, it is possible to get that $\|\tilde{v}_{\ell \ep}\|_{L^{p+1}\(\widetilde{A}_{2R(\eta)}\)} \le CR^{-2\zeta}$ by taking a smaller $\zeta$ if necessary.
Thus standard elliptic regularity theory gives
\begin{equation}\label{eq-w-bound-2}
\left|\tilde{v}_{\ell \ep}(x)\right| \le \|\tilde{v}_{\ell \ep}\|_{L^{\infty}\(B^n(x,1)\)} \le C \|\tilde{v}_{\ell \ep}\|_{L^q\(\widetilde{A}_{R-1}\)}
\le {C  \over (R-2R(\eta)-1)^{{n-2 \over 2}-\zeta}} \cdot \|\tilde{v}_{\ell \ep}\|_{L^{p+1}\(\widetilde{A}_{2R(\eta)}\)}\le {C \over R^{{n-2 \over 2}+\zeta}}
\end{equation}
for all $x \in \widetilde{A}_R$, $R \ge 3R(\eta)$.

Having \eqref{eq-w-bound-2} in mind, we now prove \eqref{eq-w-bound-0} by employing the comparison principle iteratively.
Assume that it holds
\begin{equation}\label{eq-max-0}
|\tilde{v}_{\ell \ep}(x)| \le D_j \sum_{i=1}^m \frac{1}{|x-\tilde{x}_{i \ep}|^{q_j}} \quad \text{for all } x \in \widetilde{A}_R,
\end{equation}
some $D_j > 0$ and $0 < q_j < n-2$ to be determined soon ($j \in \mathbb{N}$).
Since we have $(n-2)(p-1-\ep) > 3$ for small $\ep > 0$, Proposition \ref{prop-LZ}, Lemma \ref{lem-u-bound} and \eqref{eq-max-0} tell us that there exists some $\widetilde{D}_j > 0$ whose choice is affected by only $D_j$, $n$ and $\ell$ such that
\[-\Delta (\tilde{v}_{\ell \ep})_{\pm}(x) = \mu_{\ell \ep} (p-\ep) \tu_{\ep}^{p-1-\ep} (\tilde{v}_{\ell \ep})_{\pm}(x)
\le \widetilde{D}_j \sum_{i=1}^m \frac{1}{|x - \tilde{x}_{i\ep}|^{q_j + 3}} \quad \text{for any } x \in \widetilde{A}_R.\]
Select any number $0 < \tilde{\eta} < \min(1, n-2-q_j)$ and set a function
\[\chi_j(x) = D_{j+1} \sum_{i=1}^m \frac{1}{|x - \tilde{x}_{i\ep}|^{q_j + \tilde{\eta}}} \quad \text{for } x \in \mr^n\]
where $D_{j+1} > 0$ is a number so large that $\chi_j \ge |\tilde{v}_{\ell \ep}|$ on $\cup_{i=1}^m \pa B^n(\tilde{x}_{i\ep},R)$.
Then one can compute
\begin{equation}\label{eq-max-1}
\begin{aligned}
-\Delta \chi_j (x) &= D_{j+1} \(q_j + \tilde{\eta}\) \((n-2)-\(q_j +\tilde{\eta}\)\) \sum_{i=1}^m \frac{1}{|x-\tilde{x}_{i\ep}|^{q_j +\tilde{\eta}+2}} \\
&\ge \widetilde{D}_j \sum_{i=1}^m \frac{1}{|x-\tilde{x}_{i\ep}|^{q_j + 3}} \ge - \Delta (\tilde{v}_{\ell \ep})_{\pm}(x), \quad x \in \widetilde{A}_R
\end{aligned}
\end{equation}
by taking a larger $D_{j+1}$ if necessary.
However $\chi_j > 0$ and $\tilde{v}_{\ell \ep} = 0$ on $\pa \Omega_{\ep}$, whence $\chi_j \ge |\tilde{v}_{\ell \ep}|$ on $\pa \widetilde{A}_R$.
Consequently, by \eqref{eq-max-1} and the maximum principle, it follows that
\[|\tilde{v}_{\ell \ep}(x)| \le \chi_j(x) = D_{j+1} \sum_{i=1}^m \frac{1}{|x-\tilde{x}_{i\ep}|^{q_j+\tilde{\eta}}}, \quad x \in \widetilde{A}_R.\]
Letting $q_1 = {n-2 \over 2} + \zeta$ in \eqref{eq-w-bound-2}, choosing an appropriate $D_1 >0$ and repeating this comparison procedure, we can deduce
\[|\tilde{v}_{\ell \ep}(x)| \le C \sum_{i=1}^m \frac{1}{|x-\tilde{x}_{i\ep}|^q}, \quad x \in \widetilde{A}_R\]
given any $1 < q < n-2$.
This proves \eqref{eq-w-bound-0}.

Finally, \eqref{eq-w-bound-3} and the claim that $v_{\ep} = O(\ep)$ in $A_r$ is a straightforward consequence of \eqref{eq-w-bound-0}.
The proof is completed.
\end{proof}

By utilizing \eqref{eq-LZ}, \eqref{eq-w-bound-3}, \eqref{eq-u-bound}, the fact that $v_{\ep} = O(\ep)$ in $A_r$ and regularity theory,
we immediately establish a decay estimate for the derivatives of $\tu_{i\ep}$ and $\tilde{v}_{\ell i \ep}$.
\begin{lem}\label{lem-u-d-bound}
For any $k \in \{1, \cdots, n\}$, there exists a universal constant $C > 0$ such that
\[\left|{\pa \tu_{i\ep} (x) \over \pa x_k}\right| \le \frac{C}{1+|x|^{n-2}} \quad \text{and} \quad \left|{\pa \tilde{v}_{\ell i \ep} (x) \over \pa x_k}\right|
\le \frac{C}{1+|x|^{n-2-\zeta}} \quad \text{for all } |x| \le \ep^{-\alpha_0}r\]
for $\zeta,\ r > 0$ small. 
Moreover we have
\[\left|\pa_k u_{\ep}\right|,\ \left|\pa_k\pa_l u_{\ep}\right| = O\(\sqrt{\ep}\) \quad \text{and} \quad \left|\pa_k v_{\ell \ep}\right| = O(\ep) \quad \text{for all } k,\ l = 1, \cdots, n\]
as $\ep \to 0$ in any compact subset of $A_r$.
\end{lem}

\medskip
Finally, we recall two well-known results.
The first lemma states the nondegeneracy property of the standard bubble $U_{1,0}$. We refer to \cite{BE} for its proof.
\begin{lem}\label{lem-nondeg}
The space of solutions to the linear problem
\[-\Delta v = p U_{1,0}^{p-1} v \quad \text{in } \mr^n, \quad v \in D^{1,2}(\mr^n)\]
is spanned by
\[{x_1 \over (1+|x|^2)^{n \over 2}},\ \cdots,\ {x_n \over (1+|x|^2)^{n \over 2}} \quad \text{and} \quad {1-|x|^2 \over (1+|x|^2)^{n \over 2}}.\]
\end{lem}
\noindent The next lemma lists some formulas regarding the derivatives of Green's function. The proof can be found in \cite{GP, H}.
\begin{lem}\label{lem-Green}
For $x_0 \in \Omega$, it holds that
\[\int_{\pa \Omega} (x-x_0, \nu) \(\frac{\pa G}{\pa \nu}(x,x_0)\)^2 dS = (n-2) \tau(x_0),\]
\[\int_{\pa \Omega} \(\frac{\pa G}{\pa \nu}(x,x_0)\)^2 \nu_k (x) dS = \frac{\pa \tau}{\pa x_k}(x_0),\quad k =1, \cdots, n\]
and
\[\int_{\pa \Omega}\frac{\pa G}{\pa x_k} (x,x_0) \frac{\pa}{\pa y_l}\(\frac{\pa G}{\pa \nu}(x,x_0)\) dS =
\frac{1}{2}\frac{\pa^2 \tau}{\pa x_k \pa x_l}(x_0),\quad k,l = 1, \cdots, n.\]
Here $\nu$ is the outward normal unit vector to $\pa \Omega$ and $dS$ is the surface measure $\pa \Omega$.
\end{lem}

\section{Proof of Theorem \ref{thm-eigen-zero}}\label{sec-3}
In this section, we present estimates for the first $m$ eigenvalues and eigenfunctions of \eqref{eq-lin}.

\medskip
For the set of the concentration points $\{x_1, \cdots, x_m\} \subset \Omega^m$, let us fix a small number $r > 0$ such that for any $1 \le i \ne j \le m$ and any $\ep > 0$ small the following holds:
\[B^n(x_i, 4r) \subset \Omega \quad \text{and} \quad B^n(x_i, 4r) \cap B^n(x_j, 4r) = \emptyset.\]
For each $1 \le i \le m$, we set $\phi_i(x)= \phi(x - x_i)$ where a cut-off function $\phi \in C_c^{\infty}(B^n(0,3r))$ satisfies $\phi \equiv 1$ in $B^n(0,2r)$ and $0 \le \phi \le 1$ in $B^n(0,3r)$.
Define also
\begin{equation}\label{eq-u-psi}
u_{\ep, i} = \phi_i u_{\ep},
\quad \psi_{\ep, i, k} = \phi_i \frac{\pa u_{\ep}}{\pa x_k}\ (1 \le k \le n)
\quad \text{and} \quad \psi_{\ep, i, n+1} = \phi_i \cdot \(\(x-x_{i \ep}\) \cdot \nabla u_{\ep}+ \frac{2 u_{\ep}}{p-1-\ep}\)
\end{equation}
in $\Omega$.

\medskip
The following lemma serves as a main ingredient for the proof of Theorem \ref{thm-eigen-zero}
\begin{lem} \label{lem-weak-nonzero}
Fix $\ell \in \mathbb{N}$.
Suppose that $\{v_{\ell \ep}\}_{\ep}$ is a family of normalized eigenfunctions of \eqref{eq-lin} corresponding to the $\ell$-th eigenvalue $\mu_{\ell \ep}$.
Then there exists at least one $i_0 \in \{1,\cdots, m\}$ such that $\tilde{v}_{\ell i_0 \ep}$ (see \eqref{eq-v-exp} for its definition) converges to a nonzero function in the weak $H^1(\mr^n)$-sense.
\end{lem}
\begin{proof}
Lemma \ref{lem-w-bound} ensures that there exist a large $R > 0$ and a small $r > 0$ such that $|\tilde{v}_{\ell i \ep}| \le 1/2$ for $R \le |x| \le \ep^{-\alpha_0}r$.
Suppose that $\tilde{v}_{\ell i \ep} \rightharpoonup 0$ weakly in $H^1(\mr^n)$ as $\ep \to 0$ for all $1 \le i \le m$.
Then each $\tilde{v}_{\ell i \ep}$ tends to 0 uniformly in $B^n(0,R)$ by elliptic regularity.
Since we already know that $v_{\ep} \to 0$ uniformly on $A_r$ from Lemma \ref{lem-con-ext}, it follows that $\|v_{\ep}\|_{L^{\infty}(\Omega)} \le 1/2$.
However $\|v_{\ep}\|_{L^{\infty}(\Omega)} = 1$ by its own definition, hence a contradiction arises.
\end{proof}

\noindent Given Lemma \ref{lem-weak-nonzero}, we are now ready to start to prove Theorem \ref{thm-eigen-zero}.
\begin{proof}[Proof of Theorem \ref{thm-eigen-zero}]
Let $\mathcal{V}$ be a vector space whose basis consists of $\{u_{\ep, i}: 1 \le i \le m\}$.
By the Courant-Fischer-Weyl min-max principle, we have
\[\mu_{m \ep} = \min_{\substack{\mathcal{W} \subset H_0^1(\Omega),\\ \text{dim} \mathcal{W} = m}} \max_{f \in \mathcal{W} \setminus\{0\}}
\frac{\int_{\Omega}|\nabla f(x)|^2dx}{(p-\ep)\int_{\Omega} \(f^2 u_{\ep}^{p-1-\ep}\)(x) dx}
\le \max_{f \in \mathcal{V} \setminus\{0\}}\frac{\int_{\Omega}|\nabla f(x)|^2 dx}{(p-\ep)\int_{\Omega} \(f^2 u_{\ep}^{p-1-\ep}\)(x) dx}.\]
With this characterization in hand, it is easy to derive that $\mu_{m \ep} \le p^{-1} + o(1)$.
Thus, if we let $\mu_{\ell} = \lim\limits_{\ep \to 0} \mu_{\ell \ep}$, we know that $\mu_{\ell} \le p^{-1}$ for any $1 \le \ell \le m$.

Fix $\ell \in \{1, \cdots, m\}$.
By Lemma \ref{lem-weak-nonzero} there is an index $i_0 \in \{1, \cdots, m\}$ such that $\tilde{v}_{\ell i_0 \ep}$ converges $H^1(\mr^n)$-weakly to a nonzero function $V$.
A direct computation shows
\[-\Delta \tilde{v}_{\ell i_0 \ep} = \mu_{\ell \ep} (p-\ep) \tu_{i \ep}^{p-1-\ep}\tilde{v}_{\ell i_0 \ep} \quad \text{in } \Omega_{i_0 \ep}\]
where the function $\tu_{i \ep}$ and the set $\Omega_{i_0 \ep}$ are defined in \eqref{eq-u-exp} and \eqref{eq-v-exp}, respectively.
Thus it follows from Lemma \ref{lem-u-dil} that $V \in H^1(\mr^n) \setminus \{0\}$ is a solution of
\[-\Delta V = \mu_{\ell} p U_{1,0}^{p-1} V \quad \text{in } \mr^n.\]
Consequently, the estimate for $\mu_{\ell}$ in the previous paragraph implies that $\mu_{\ell} = p^{-1}$.

On the other hand, for any $i$, we also see that $\tilde{v}_{\ell i \ep}$ converges to a function $W$ weakly in $H^1(\mr^n)$ so that $W$ solves $-\Delta W = U^{p-1}_{1,0} W$ in $\mr^n$.
Thus there is a nonzero vector $\mathbf{c}_{\ell}= \(\lambda_1^{n-2 \over 2} c_{\ell 1}, \cdots, \lambda_m^{n-2 \over 2} c_{\ell m}\) \in \mr^m$
such that $\tilde{v}_{\ell i \ep} \rightharpoonup c_{\ell i} U_{1,0}$ weakly in $H^1(\mr^n)$ for each $i \in \{1,\cdots, m\}$.

Let us prove \eqref{eq-eigen-zero-0} now.
Fixing $i$, we multiply \eqref{eq-main-e} (or \eqref{eq-lin} with $v = v_{\ell \ep}$) by $v_{\ell \ep}$ (or $u_{\ep}$) to get the identity, say, $I$ (or $II$ respectively).
Also we denote by $\int I$ and $\int II$ the identities which can be obtained after integrating $I$ and $II$ over $B^n(x_{i\ep}, r)$.
Subtracting $\int I$ from $\int II$ and utilizing Green's identity \eqref{eq-green-id} below, we see then
\begin{equation}\label{eq-eigen-zero-1}
\int_{\pa B^n(x_{i\ep},r)} \(\frac{\pa u_{\ep}}{\pa \nu} v_{\ell \ep} - \frac{\pa v_{\ell \ep}}{\pa \nu} u_{\ep}\) dS = \(\mu_{\ell\ep}(p-\ep)-1\) \int_{B^n(x_{i\ep},r)} \(u_{\ep}^{p-\ep} v_{\ell \ep}\)(x) dx
\end{equation}
for each $i \in \{1, \cdots, m\}$ and any $r > 0$ sufficiently small.
Moreover, if we set the functions
\[C_2^{-1}\tilde{g}_i(x) = - \lambda_i^{n-2 \over 2} H(x,x_i) + \sum_{j \ne i} \lambda_j^{n-2 \over 2} G(x,x_j), \quad C_2^{-1}\tilde{h}_i(x) = - \lambda_i^{n-2} c_{\ell i} H(x,x_i) + \sum_{j \ne i} \lambda_j^{n-2} c_{\ell j} G(x,x_j)\]
which are harmonic near $x_i$, then (the proof of) Lemma \ref{lem-u-asym} permits us to obtain that
\begin{equation}\label{eq-eigen-zero-2}
\ep^{-{1 \over 2}} u_{\ep}(x) = C_2 \lambda_i^{n-2 \over 2}\frac{\gamma_n}{|x-x_i|^{n-2}} + \tilde{g}_i(x) + o(1)
\end{equation}
and
\begin{equation}\label{eq-eigen-zero-3}
\ep^{-1} {v_{\ell \ep}(x) \over \mu_{\ell\ep}(p-\ep)} = C_2 \lambda_i^{n-2} c_{\ell i} \frac{\gamma_n}{|x-x_i|^{n-2}} + \tilde{h}_i(x) + o(1)
\end{equation}
for $x \in B^n(x_{i\ep},2r)$.
Therefore, by inserting \eqref{eq-eigen-zero-2} and \eqref{eq-eigen-zero-3} into \eqref{eq-eigen-zero-1},
and using the mean value formula for harmonic functions and $\nabla_{\lambda} \Upsilon(\lambda_1, \cdots, \lambda_m, x_1, \cdots, x_m) = 0$ then, one discovers
\begin{align*}
& \int_{\pa B^n(x_{i\ep},r)} \left[\frac{\pa \(\ep^{-{1 \over 2}}u_{\ep}\)}{\pa \nu} \(\ep^{-1} v_{\ell \ep}\) - \frac{\pa \(\ep^{-1} v_{\ell \ep}\)}{\pa \nu} \(\ep^{-{1 \over 2}}u_{\ep}\)\right] dS \\
&\qquad \to (n-2)C_2\gamma_n\left|S^{n-1}\right| \(\lambda_i^{n-2}c_{\ell i} \tilde{g}_i(x_i) - \lambda_i^{n-2 \over 2} \tilde{h}_i(x_i)\) \\
&\qquad = c_1 \left[ \lambda_i^{n-2} \(\sum_{j \ne i} \lambda_j^{n-2 \over 2} G(x_i,x_j)\) c_{\ell i} - \sum_{j \ne i} \lambda_i^{n-2 \over 2}\lambda_j^{n-2} G(x_i,x_j) c_{\ell j} \right] \\
&\qquad = \(c_1 \lambda_i^{3(n-2) \over 2} \tau(x_i) - {c_2 \over n-2} \lambda_i^{n-2 \over 2}\) c_{\ell i} - c_1\sum_{j \ne i} \lambda_i^{n-2 \over 2}\lambda_j^{n-2} G(x_i,x_j) c_{\ell j}
\end{align*}
and
\[\ep^{-{1 \over 2}}\int_{B^n(x_{i\ep},r)} u_{\ep}^{p-\ep} v_{\ell \ep} \to c_{\ell i} \lambda_i^{\frac{n-2}{2}} {4n c_2 \over (n-2)^2}\]
(refer to \eqref{eq-upsilon-2}). From this, we get
\begin{multline*}
\(\lambda_i^{n-2} \tau(x_i) - {c_2 \over (n-2)c_1}\) \(\lambda_i^{n-2 \over 2}c_{\ell i}\) - \sum_{j \ne i} (\lambda_i\lambda_j)^{n-2 \over 2} G(x_i,x_j) \(\lambda_j^{n-2 \over 2}c_{\ell j}\) \\
= \({4nc_2 \over (n-2)^2c_1}\) \cdot \lim_{\ep \to 0} \({\mu_{\ell\ep}(p-\ep)-1 \over \ep}\) \(\lambda_i^{\frac{n-2}{2}}c_{\ell i}\) := \rho^1_{\ell} \(\lambda_i^{\frac{n-2}{2}}c_{\ell i}\),
\end{multline*}
or equivalently, $\ma_1 \mathbf{c}_{\ell} = \rho^1_{\ell} \mathbf{c}_{\ell}$.
This justifies \eqref{eq-eigen-zero-0}.
We also showed that $\mathbf{c}_{\ell}^T$ is an eigenvector corresponding to the eigenvalue $\rho^1_{\ell}$ at the same time.

Finally, to verify the last assertion of the theorem, we assume that $\ell_1 \ne \ell_2$.
Since $v_{\ell_1 \ep}$ and $v_{\ell_2 \ep}$ are orthogonal each other, we have
\begin{equation}\label{eq-eigen-zero-4}
\begin{aligned}
0 &= \lim_{\ep \to 0} \ep^{-1} \(\mu_{\ell_1 \ep}(p-\ep)\)^{-1} \int_{\Omega} \nabla v_{\ell_1 \ep} \cdot \nabla v_{\ell_2 \ep}
\\
&= \lim_{\ep \to 0} \ep^{-1} \(\sum_{i=1}^m \int_{B^n(x_{i \ep}, r)} u_{\ep}^{p-1-\ep} v_{\ell_1 \ep} v_{\ell_2 \ep} + \int_{\Omega \setminus \cup_{i=1}^{m} B^n(x_{i \ep}, r)} u_{\ep}^{p-1-\ep} v_{\ell_1 \ep} v_{\ell_2 \ep}\)
\\
&= \lim_{\ep \to 0} \sum_{i=1}^m \lambda_{i \ep}^{n-2}\int_{B^n\(0, \(\lambda_{i \ep}\ep^{\alpha_0}\)^{-1}r\)} \tu_{i \ep}^{p-1-\ep} \tilde{v}_{\ell_1 i \ep} \tilde{v}_{\ell_2 i \ep}
= \sum_{i=1}^m \(\lambda_i^{n-2 \over 2}c_{\ell_1 i}\) \(\lambda_i^{n-2 \over 2}c_{\ell_2 i}\) \int_{\mr^n} U_{1,0}^{p+1}.
\end{aligned}
\end{equation}
Thus $\textbf{c}_{\ell_1}^T \cdot \textbf{c}_{\ell_2}^T = 0$.
Here the last equality can be justified by the dominated convergence theorem with Proposition \ref{prop-LZ} and Lemma \ref{lem-w-bound}.
\end{proof}

\section{Upper bounds for the $\ell$-th eigenvalues and asymptotic behavior of the $\ell$-th eigenfunctions, $m+1 \le \ell \le (n+1)m$}\label{sec-4}
The objective of this section is to provide estimates of the $\ell$-th eigenvalues and its corresponding eigenfunctions when $m+1 \le \ell \le (n+1)m$.
Their refinement will be accomplished in the subsequent sections based on the results deduced in this section.

\medskip
In the first half of this section, our interest will lie on achieving upper bounds of the eigenvalues $\mu_{\ell \ep}$ for $m+1 \le \ell \le (n+1)m$, as the following proposition depicts.
\begin{prop}\label{prop-apriori}
Suppose that $m+1 \le \ell \le (n+1)m$. Then
\[\mu_{\ell \ep} \le 1 + O\(\ep^{\frac{n}{n-2}}\).\]
\end{prop}
\begin{proof}
We define a linear space $\mathcal{V}$ spanned by
\[\{u_{\ep,i}: 1 \le i \le m\} \cup \{\psi_{\ep, i, k}: 1 \le i \le m,\ 1 \le k \le n\}\]
(refer to \eqref{eq-u-psi}).
By the variational characterization of the eigenvalue $\mu_{\ell \ep}$, we have
\[\mu_{((n+1)m) \ep} = \min_{\substack{\mathcal{W} \subset H_0^1(\Omega),\\ \text{dim}\mathcal{W} = (n+1)m}} \max_{f \in \mathcal{W} \setminus\{0\}}
\frac{\int_{\Omega}|\nabla f|^2}{(p-\ep)\int_{\Omega} f^2 u_{\ep}^{p-1-\ep}}
\le \max_{f \in \mathcal{V} \setminus\{0\}} \frac{\int_{\Omega}|\nabla f|^2}{(p-\ep)\int_{\Omega} f^2 u_{\ep}^{p-1-\ep}}.\]
Any nonzero function $f \in \mathcal{V} \setminus \{0\}$ can be written as
\[f = \sum_{i=1}^m f_i \quad \text{with} \quad f_i = a_{i0} u_{\ep, i} + \sum_{k=1}^n a_{ik} \psi_{\ep,i,k}\]
for some nonzero numbers $a_{ik}$ ($1 \le i \le m$ and $0 \le k \le n$).
As $f_{i_1}$ and $f_{i_2}$ have disjoint supports if $1 \le i_1 \ne i_2 \le m$,
\[\frac{\int_{\Omega} |\nabla f|^2}{(p-\ep) \int_{\Omega} f^2 u_{\ep}^{p-1-\ep}} \le \max_{1\le i \le m}
\frac{\int_{\Omega} |\nabla f_i|^2}{(p-\ep)\int_{\Omega} f_i^2 u_{\ep}^{p-1-\ep}} := \max_{1 \le i \le m} \mathfrak{a}_i.\]
Hence it suffices to show that each $\mathfrak{a}_i$ is bounded by $1+O\(\ep^{n \over n-2}\)$.
As a matter of fact, this can be achieved along the line of the proof of \cite[Proposition 3.2]{GP},
but we provide a brief sketch here since our argument slightly simplifies the known proof.

Fix $i \in \{1, \cdots, m\}$.
For the sake of notational simplicity, we write $\mathfrak{a} = \mathfrak{a}_i$, $\phi = \phi_i$ and $a_k = a_{ik}$ for $0 \le k \le n$.
Denote also $z_{\ep} = \sum_{k=1}^n a_k {\pa u_{\ep} \over \pa x_k}$ so that $f_i = a_0 \phi u_{\ep} + \phi z_{\ep}$.
After multiplying \eqref{eq-main-e} by $\phi^2 u_{\ep}$ or $\phi^2 z_{\ep}$, and integrating the both sides over $\Omega$, one can deduce
\begin{equation}\label{eq-apriori-1}
\int_{\Omega} |\nabla(\phi u_{\ep})|^2 = \int_{\Omega} |\nabla \phi|^2 u_{\ep}^2 +  \int_{\Omega} \phi^2 u_{\ep}^{p+1-\ep}.
\end{equation}
and
\begin{equation}\label{eq-apriori-2}
\int_{\Omega} \nabla(\phi u_{\ep}) \cdot \nabla (\phi z_{\ep}) = \int_{\Omega} |\nabla \phi|^2 u_{\ep} z_{\ep} + \int_{\Omega} \phi \nabla \phi \cdot \(u_{\ep} \nabla z_{\ep}
- z_{\ep} \nabla u_{\ep}\) + \int_{\Omega} \phi^2 u_{\ep}^{p-\ep} z_{\ep}.
\end{equation}
Similarly, testing $-\Delta z_{\ep} = (p-\ep)u_{\ep}^{p-1-\ep} z_{\ep}$ with $\phi^2 z_{\ep}$, one finds that
\begin{equation}\label{eq-apriori-3}
\int_{\Omega}|\nabla (\phi z_{\ep})|^2 = \int_{\Omega}|\nabla \phi|^2 z_{\ep}^2 + (p-\ep) \int_{\Omega} \phi^2 u_{\ep}^{p-1-\ep} z_{\ep}^2.
\end{equation}
Then \eqref{eq-apriori-1}-\eqref{eq-apriori-3} yields $\mathfrak{a} = 1 + \mathfrak{b}/\mathfrak{c}$ where
\begin{multline}\label{eq-b}
\mathfrak{b} = -(p-1-\ep) \(a_0^2 \int_{\Omega} \phi^2 u_{\ep}^{p+1-\ep} + 2a_0 \int_{\Omega} \phi^2 u_{\ep}^{p-\ep} z_{\ep}\) + a_0^2 \int_{\Omega} |\nabla \phi|^2 u_{\ep}^2
\\
+ \int_{\Omega} |\nabla \phi|^2 z_{\ep}^2 + 2a_0 \int_{\Omega} \phi \nabla \phi \cdot \(u_{\ep} \nabla z_{\ep} - z_{\ep} \nabla u_{\ep}\) + 2a_0 \int_{\Omega} |\nabla \phi|^2 u_{\ep} z_{\ep}.
\end{multline}
and
\begin{equation}\label{eq-c}
\mathfrak{c} = (p-\ep) \(a_0^2 \int_{\Omega} \phi^2 u_{\ep}^{p+1-\ep} + 2 a_0 \int_{\Omega} \phi^2 u_{\ep}^{p-\ep} z_{\ep} + \int_{\Omega} \phi^2 u_{\ep}^{p-1-\ep} z_{\ep}^2 \).
\end{equation}
Our aim is to find an upper bound of $\mathfrak{b}$ and a lower bound of $\mathfrak{c}$.
Let us estimate $\mathfrak{b}$ first. We see at once that
\[-(p-1-\ep) a_0^2 \int_{\Omega} \phi^2 u_{\ep}^{p+1-\ep} < -C a_0^2.\]
Also, if we let $\bar{a} =(a_1,\cdots, a_n)$, then \eqref{eq-u-bound} guarantees
\[\left| a_0 \int_{\Omega} \phi^2 u_{\ep}^{p-\ep} z_{\ep} \right|
= \left| a_0 \sum_{j=1}^n a_k \int_{\Omega} \phi^2  u_{\ep}^{p-\ep} \frac{\pa u_{\ep}}{\pa x_k} \right|
= \left| \frac{a_0}{p+1-\ep} \sum_{k=1}^n a_k \int_{\Omega}  \frac{\pa \phi^2}{\pa x_k} u_{\ep}^{p+1-\ep} \right| \le C a_0 |\bar{a}|\ep^{p+1-\ep \over 2}.\]
Moreover we have that
\[a_0^2 \int_{\Omega} |\nabla \phi|^2 u_{\ep}^2 \le C a_0^2 \ep.\]
On the other hand, for $\mathcal{D}_1 = B^n(x_i, 3r)\setminus B^n(x_i, 2r)$ and $\mathcal{D}_2 =B^n(x_i,4r)\setminus B^n(x_i, r)$, we easily discover
\[\int_{\Omega} |\nabla \phi|^2 z_{\ep}^2 \le C \int_{\mathcal{D}_1} z_{\ep}^2 \le C|\bar{a}|^2 \int_{\mathcal{D}_1} |\nabla u_{\ep}|^2 \le C |\bar{a}|^2 \int_{\mathcal{D}_2} \(u_{\ep}^{p+1-\ep} + u_{\ep}^2\) \le C |\bar{a}|^2 \ep\]
and
\[\int_{\mathcal{D}_1} |\nabla z_{\ep}|^2 \le C \int_{\mathcal{D}_2} \(z_\ep^2 + u_{\ep}^{p-1-\ep} z_{\ep}^2\) \le C \int_{\mathcal{D}_2} z_{\ep}^2 \le C |\bar{a}|^2 \ep\]
(cf. \eqref{eq-apriori-1} and \eqref{eq-apriori-3}), which implies
\[\left|2a_0 \int_{\Omega} \phi \nabla \phi \cdot \(u_{\ep} \nabla z_{\ep} - z_{\ep} \nabla u_{\ep}\) + 2a_0 \int_{\Omega} |\nabla \phi|^2 u_{\ep} z_{\ep}\right| \le C a_0 |\bar{a}| \ep.\]
Utilizing these estimates and the Cauchy-Schwarz inequality we deduce
\begin{equation}\label{eq-apriori-4}
\mathfrak{b} \le C |\bar{a}|^2 \ep.
\end{equation}
To obtain a lower bound of $\mathfrak{c}$, we note that
\[\left|\int_{\Omega} \phi^2 u_{\ep}^{p-\ep} \frac{\pa u_{\ep}}{\pa x_k} \right| = \left| \frac{1}{p+1-\ep} \int_{\Omega} \frac{\pa \phi^2}{\pa x_k} u_{\ep}^{p+1-\ep} \right| \le C \ep^{p+1-\ep \over 2}\]
and that Lemma \ref{lem-u-d-bound} ensures
\[\int_{\Omega} \phi^2 u_{\ep}^{p-1-\ep} \frac{\pa u_{\ep}}{\pa x_k}\frac{\pa u_{\ep}}{\pa x_l} = \lambda_{i\ep}^{-2}\ep^{-{2 \over n-2}} \(\frac{\delta_{kl}}{n} \int_{\mr^n} U_{1,0}^{p-1} |\nabla U_{1,0}|^2 + o(1)\)\]
for $1 \le k,\ l \le n$. Hence we conclude that
\begin{equation}\label{eq-apriori-5}
\mathfrak{c} \ge C a_0^2 - C a_0|\bar{a}| \ep^{p+1-\ep \over 2} + C|\bar{a}|^2 \ep^{-{2 \over n-2}} \ge \frac{C}{2} |\bar{a}|^2 \ep^{-{2 \over n-2}}.
\end{equation}
Consequently, a combination of \eqref{eq-apriori-4} and \eqref{eq-apriori-5} asserts that $\mathfrak{a} \le 1 + O\(\ep^{n \over n-2}\)$.
This completes the proof of the lemma.
\end{proof}

\begin{cor}\label{cor-apriori}
For $m+1 \le \ell \le (n+1)m$, we have the following limit
\[\lim_{\ep \to 0} \mu_{\ell \ep} = 1.\]
\end{cor}
\begin{proof}
By Lemma \ref{lem-weak-nonzero} we can find $i_1 \in \{1, \cdots, m\}$ such that $\tilde{v}_{\ell i_1 \ep}$ converges weakly to a nonzero function $V$.
Then, as in the proof of Theorem \ref{thm-eigen-zero}, we observe that $V$ solves
\[-\Delta V = \mu_{\ell} p U_{1,0}^{p-1} V \quad \text{in } \mr^n\]
where $\mu_{\ell} =\lim_{\ep \to 0} \mu_{\ell \ep}$.
Also, owing to Proposition \ref{prop-apriori}, we have $\mu_{\ell} \le 1$.
Since the morse index of $U_{1,0}$ is 1, it should hold that $\mu_{\ell} = p^{-1}$ or 1.

Assume that $\mu_{\ell} = p^{-1}$.
Then the proof of Theorem \ref{thm-eigen-zero} again gives us that there is a vector $\mathbf{b}_{\ell} = \(\lambda_1^{n-2 \over 2} b_{\ell 1}, \cdots, \lambda_m^{n-2 \over 2}b_{\ell m}\) \ne 0$
such that $\tilde{v}_{\ell i \ep} \rightharpoonup b_{\ell i} U_{1,0}$ weakly in $H^1(\mr^n)$.
Furthermore $\mathbf{b}_{\ell} \cdot \mathbf{c}_{\ell_1} = 0$ for any $1 \le \ell_1 \le m$, but this is impossible since $\{\mathbf{c}_1, \cdots, \mathbf{c}_m\}$ already spans $\mr^m$.
Hence $\mu_{\ell} = 1$, which finishes the proof.
\end{proof}

Next, we provide a general convergence result of the $\ell$-th $L^{\infty}(\Omega)$-normalized eigenfunction $v_{\ell \ep}$.
We recall its dilation $\tilde{v}_{\ell i \ep}$ defined in \eqref{eq-v-exp}.
\begin{lem}\label{lem-weak-conv}
Suppose that $m+1 \le \ell \le (n+1) m$.
\begin{enumerate}
\item For any $i \in \{1,\cdots, m\}$ there exists a vector $(d_{\ell,i,1}, \cdots, d_{\ell,i,n+1}) \in \mr^{n+1}$ such that the function $\tilde{v}_{\ell i \ep}$ converges to
\[\sum_{k=1}^n d_{\ell,i,k} \(\frac{\pa U_{1,0}}{\pa (x_0)_k}\) + d_{\ell,i,n+1} \(\frac{\pa U_{1,0}}{\pa \lambda}\)\]
weakly in $H^1(\mr^n)$ (see \eqref{eq-U} for the definition of $U_{\lambda, x_0}$).
In addition, there is at least one $i_1 \in \{1,\cdots,m\}$ such that $(d_{\ell,i_1,1}, \cdots, d_{\ell,i_1,n+1}) \ne 0$.
\item As $\ep \to 0$ we have
\begin{equation}\label{eq-v-conv}
\ep^{-1} v_{\ell \ep} \to C_3 \sum_{i=1}^m d_{\ell, i, n+1} \lambda_i^{n-2} G(\cdot, x_i) \quad \text{in } C^1(\Omega \setminus \{x_1, \cdots, x_m\})
\end{equation}
where $C_3 = p\int_{\mr^n} U_{1,0}^{p-1} \(\frac{\pa U_{1,0}}{\pa \lambda}\) > 0$.
\end{enumerate}
\end{lem}
\begin{proof}
It is not hard to show the first statement with Lemmas \ref{lem-weak-nonzero} and \ref{lem-nondeg}, and Corollary \ref{cor-apriori}.
Hence let us consider the second statement.
For $r > 0$ fixed small, assume that a point $x \in \Omega$ belongs to $A_r$ where $A_r$ is the set in \eqref{eq-A_r}.
According to Green's representation formula and Lemmas \ref{lem-u-bound} and \ref{lem-con-ext},
\[\ep^{-1} v_{\ell \ep}(x) = \ep^{-1} \mu_{\ell \ep}(p-\ep) \sum_{i=1}^m \int_{B^n(x_{i \ep}, r/2)}G(x,y)u_{\ep}^{p-1-\ep}(y) v_{\ell \ep}(y)dy + o(1).\]
Besides, Proposition \ref{prop-LZ} with Lemmas \ref{lem-w-bound} and \ref{lem-weak-conv} (1) allow us to obtain
\begin{equation}\label{eq-weak-conv-a}
\begin{aligned}
&\lim_{\ep \to 0} \ep^{-1} \int_{B^n(x_{i \ep}, r/2)}G(x,y)u_{\ep}^{p-1-\ep}(y) v_{\ell \ep}(y)dy \\
&\qquad = \lambda_i^{n-2} \lim_{\ep \to 0} \int_{B^n\(0, \(\lambda_{i \ep}\ep^{\alpha_0}\)^{-1}r/2\)} G\(x, x_{i\ep}+\lambda_{i\ep}\ep^{\alpha_0}y\)\(\tu_{i \ep}^{p-1-\ep}\tilde{v}_{\ell i \ep}\)(y)dy \\
&\qquad = d_{\ell,i,n+1} \lambda_i^{n-2} G(x, x_i) \int_{\mr^n}  U_{1,0}^{p-1}(y) \(\frac{\pa U_{1,0}}{\pa \lambda}\)(y) dy.
\end{aligned}
\end{equation}
Thus the lemma is proved.
\end{proof}

In fact, we can refine the first statement of the above lemma to arrive at \eqref{eq-eigenf-first-1}, which is the main result of the latter part of this section.
\begin{prop}\label{prop-weak-conv}
Let $m+1 \le \ell \le (n+1) m$.
For each $i \in \{1,\cdots,m\}$ and $(d_{\ell,i,1}, \cdots, d_{\ell,i,n}) \in \mr^n$, the function $\tilde{v}_{\ell i \ep}$ converges to
\[\sum_{k=1}^n d_{\ell,i,k} \(\frac{\pa U_{1,0}}{\pa (x_0)_k}\) = -\sum_{k=1}^n d_{\ell,i,k} \(\frac{\pa U_{1,0}}{\pa x_k}\)\]
weakly in $H^1(\mr^n)$.
\end{prop}
\noindent As a preparation for its proof, we first consider the following auxiliary lemma.
\begin{lem}\label{lem-ijl}
Fix $1 \le i \le m$. For any small $r > 0$ and $1 \le j, l \le m$, we define
\begin{multline}\label{eq-r}
\mathcal{I}_{jl}^r = \mathcal{I}_{jl; i}^r = \int_{\pa B^n(x_i, r)} \(\frac{\pa}{\pa \nu} \left[(x-x_i) \cdot \nabla G(x,x_j) + \({n-2 \over 2}\) G(x,x_j)\right] G(x,x_l)\right. \\
\left.-\left[(x-x_i) \cdot \nabla G(x,x_j) + \({n-2 \over 2}\) G(x,x_j)\right] \frac{\pa}{\pa \nu} G(x,x_l)\) dS.
\end{multline}
Then $\mathcal{I}_{jl}^r$ is independent of $r > 0$ and its value is computed as
\begin{equation}\label{eq-ijl}
\mathcal{I}_{jl}^r = \begin{cases}
0 &\text{if } j \ne i \text{ and } l \ne i,\\
\(\dfrac{n-2}{2}\) G(x_i, x_j) &\text{if } j \ne i \text{ and } l = i,\\
\(\dfrac{n-2}{2}\) G(x_i, x_l) &\text{if } j = i \text{ and } l \ne i,\\
-(n-2)\tau(x_i) &\text{if } j = l = i.
\end{cases}\end{equation}
\end{lem}
\begin{proof}
Assuming $0 < r_2 < r_1$ are small enough and putting $f(x) = (x-x_i) \cdot \nabla G(x,x_j) + G(x,x_j)$, $g(x) = G(x,x_l)$ and $D = B^n(x_i, r_1) \setminus B^n(x_i, r_2)$ into Green's identity
\begin{equation}\label{eq-green-id}
\int_{\pa D} \(\frac{\pa f}{\pa \nu} g - \frac{\pa g}{\pa \nu} f\) dS = \int_D \(\Delta f \cdot g - \Delta g \cdot f\) dx,
\end{equation}
we see that $\mathcal{I}_{jl}^r$ is constant because
\begin{equation}\label{eq-green-zero}
\Delta \left[(x-x_i) \cdot \nabla G(x,x_j) + \({n-2 \over 2}\) G(x,x_j)\right] = \Delta G(x, x_l) = 0
\end{equation}
for all $x \ne x_j,\ x_l$. Thus it suffices to find the value $\mathcal{I}_{jl} = \lim_{r \to 0} \mathcal{I}_{jl}^r$.

\medskip \noindent (1) If $j,\ l \ne i$, then $\mathcal{I}_{jl} = 0$.
This follows simply by applying \eqref{eq-green-id} for $D = B^n(x_i, r)$ since \eqref{eq-green-zero} holds for any $x \in B^n(x_i, r)$.

\medskip \noindent (2) If $j \ne i$ and $l = i$, then we have
\begin{align*}
\mathcal{I}_{jl} = \mathcal{I}_{ji} &= \lim_{r \to 0} \int_{\pa B^n(x_i, r)} - \({n-2 \over 2}\) G(x,x_j) \frac{\pa}{\pa \nu} G(x,x_i) dS \\
&= \lim_{r \to 0} \int_{\pa B^n(x_i,r)} \({n-2 \over 2}\) G(x,x_j) \cdot \frac{n-2}{(n-2)\left|S^{n-1}\right||x-x_i|^{n-1}} dS = \({n-2 \over 2}\) G(x_i, x_j).
\end{align*}

\medskip \noindent (3) Suppose that $j = i$ and $l \ne i$. In this case, we deduce
\begin{align*}
\mathcal{I}_{jl} = \mathcal{I}_{il} &= \lim_{r \to 0} \int_{\pa B^n(x_i, r)} \frac{\pa}{\pa \nu} \left[(x-x_i) \cdot \nabla G(x,x_i) + \({n-2 \over 2}\) G(x,x_i)\right] G(x,x_l) dS \\
&= \lim_{r \to 0}\int_{\pa B^n(x_i, r)} \frac{n-2}{2\left|S^{n-1}\right||x-x_i|^{n-1}} \cdot G(x,x_l) dS = \({n-2 \over 2}\) G(x_i, x_l).
\end{align*}

\medskip \noindent (4) If $k=l=j$, then the Green's identity, the fact that $G(x,x_i) = 0$ on $\pa \Omega$ and Lemma \ref{lem-Green} lead
\begin{align*}
\mathcal{I}_{jl} = \mathcal{I}_{ii} &= \int_{\pa \Omega} \( \frac{\pa}{\pa \nu} \left[(x-x_i) \cdot \nabla G(x,x_i) + \({n-2 \over 2}\) G(x,x_i)\right] G(x,x_i)\right. \\
&\hspace{70pt} - \left. \left[(x-x_i)\cdot \nabla G(x,x_i) + \({n-2 \over 2}\) G(x,x_i)\right] \frac{\pa}{\pa \nu} G(x,x_i)\) dS \\
&= - \int_{\pa \Omega} \left[(x-x_i) \cdot \nabla G(x,x_i)\right] \frac{\pa}{\pa \nu} G(x,x_i) dS = -(n-2) \tau(x_i).
\end{align*}
All the computations made in (1)-(4) show the validity of \eqref{eq-ijl}.
\end{proof}

\begin{proof}[Proof of Proposition \ref{prop-weak-conv}]
Fix $i \in \{1, \cdots, m\}$ and let
\begin{equation}\label{eq-w_e}
w_{i \ep}(x) = (x-x_{i \ep}) \cdot \nabla u_{\ep} + \frac{2u_{\ep}}{p-1-\ep} \quad \text{for } x \in \Omega,
\end{equation}
a solution of
\[-\Delta w_{i \ep} = (p-\ep) u_{\ep}^{p-\ep-1} w_{i \ep} \quad \text{in } \Omega.\]
Then by \eqref{eq-green-id} it satisfies
\begin{equation}\label{eq-w-v}
\int_{\pa B^n(x_{i\ep}, r)} \(\frac{\pa w_{i \ep}}{\pa \nu} v_{\ell \ep} - \frac{\pa v_{\ell \ep}}{\pa \nu} w_{i \ep} \) dS
= (\mu_{\ell \ep} - 1) (p-\ep) \int_{B^n(x_{i\ep}, r)} u_{\ep}^{p-1-\ep} w_{i \ep} v_{\ell \ep}
\end{equation}
for $r > 0$ small, where $\nu$ is the outward normal unit vector to the sphere $\pa B^n(x_i, r)$.

In light of Lemma \ref{lem-weak-conv} (1), we only need to verify that $d_{\ell, i, n+1} = 0$ for all $i \in \{1, \cdots, m\}$.
Assume to the contrary that $d_{\ell, i, n+1} = 0$ for some $i$.
We will achieve a contradiction by showing that an estimate of $\mu_{l \ep}-1$ obtained through \eqref{eq-w-v} does not match to one found in Proposition \ref{prop-apriori}.
To reduce the notational complexity, we use $d_i$ or $d_{\ell, i}$ to denote $d_{\ell, i, n+1}$ in this proof.

Let us observe from Lemma \ref{lem-u-asym} and \eqref{eq-w_e} that
\begin{equation}\label{eq-w-asym}
\ep^{-{1 \over 2}} w_{i \ep}(x) \to C_2 \sum_{j=1}^m \lambda_j^{\frac{n-2}{2}} \left[(x-x_i) \cdot \nabla G(x,x_j) + \({n-2 \over 2}\) G(x,x_j)\right] \quad \text{in } C^1(\Omega \setminus \{x_1, \cdots, x_m\})
\end{equation}
as $\ep \to 0$. Combining this with \eqref{eq-v-conv} we get
\[\lim_{\ep \to 0} \ep^{-{3 \over 2}} \int_{\pa B^n(x_{i\ep}, r)} \(\frac{\pa w_{i \ep}}{\pa \nu} v_{\ell \ep} - \frac{\pa v_{\ell \ep}}{\pa \nu} w_{i \ep}\)dS
= C_2C_3 \sum_{j,l=1}^m \lambda_j^{n-2 \over 2}\lambda_l^{n-2} d_l \mathcal{I}^r_{jl}\]
where $\mathcal{I}^r_{jl}$ is the value defined in \eqref{eq-r}.
By inserting \eqref{eq-ijl} into the above identity, we further find that
\begin{align*}
&\lim_{\ep \to 0} \ep^{-{3 \over 2}} \int_{\pa B^n(x_{i\ep}, r)} \(\frac{\pa w_{i \ep}}{\pa \nu} v_{\ell \ep} - \frac{\pa v_{\ell \ep}}{\pa \nu} w_{i \ep}\)dS \\
&= C_2C_3 \left[\({n-2 \over 2}\) \lambda_i^{n-2 \over 2} \sum_{l \ne i} \lambda_l^{n-2} d_l G(x_i,x_l)
+ \lambda_i^{n-2} d_i \(\sum_{j \ne i} \({n-2 \over 2}\) \lambda_j^{n-2 \over 2} G(x_i,x_j) - (n-2) \lambda_i^{n-2 \over 2} \tau(x_i)\) \right] \\
&= C_2C_3 \left[\({n-2 \over 2}\) \lambda_i^{n-2 \over 2} \sum_{j \ne i} \lambda_j^{n-2} d_j G(x_i,x_j)
- \lambda_i^{n-2} d_i \({n-2 \over 2}\)\(\lambda_i^{n-2 \over 2} \tau(x_i) + C_0\lambda_i^{-{n-2 \over 2}}\) \right].
\end{align*}
Here $C_0 = c_2/((n-2)c_1) > 0$ as in \eqref{eq-mtx-M_0},
and we employed the fact that $(\lambda_1, \cdots, \lambda_m, x_1, \cdots, x_m)$ is a critical point of the functional $\Upsilon_m$ (see \eqref{eq-upsilon}) so as to obtain the second equality.
Borrowing the notation of the matrix $\ma_3$ in \eqref{eq-mtx-A_3}, the left-hand side of \eqref{eq-w-v} can be described in a legible way.
\begin{equation}\label{eq-weak-conv-1}
\lim_{\ep \to 0} \ep^{-{3 \over 2}} \int_{\pa B^n(x_{i\ep}, r)} \(\frac{\pa w_{i \ep}}{\pa \nu} v_{\ell \ep} - \frac{\pa v_{\ell \ep}}{\pa \nu} w_{i \ep}\)dS
= - C_2C_3 \({n-2 \over 2}\) \sum_{j=1}^m \ma^3_{ij} \(\lambda_j^{n-2 \over 2}d_j\).
\end{equation}
On the other hand, counting on Proposition \ref{prop-LZ} and Lemmas \ref{lem-u-dil} and \ref{lem-weak-conv}, we can compute its right-hand side as follows.
\begin{equation}\label{eq-weak-conv-2}
\begin{aligned}
&\lim_{\ep \to 0} \ep^{-{1 \over 2}} \int_{B^n(x_{i\ep}, r)} u_{\ep}^{p-1-\ep}(x) \left[(x-x_{i\ep}) \cdot \nabla u_{\ep}(x) + \frac{2u_{\ep}(x)}{p-1-\ep} \right] v_{\ell \ep}(x) dx
\\
&\qquad = \lim_{\ep \to 0} \lambda_i^{n-2 \over 2} \int_{B^n\(0,\(\lambda_i\ep^{\alpha_0}\)^{-1}r\)} \tu_{i \ep}^{p-1-\ep}(y) \left[y \cdot \nabla \tu_{i \ep}(y) + {2 \tu_{i \ep}(y) \over p-1-\ep} \right] \tilde{v}_{\ell i \ep}(y) dy
\\
&\qquad = \lambda_i^{n-2 \over 2} d_i \int_{\mr^n} U_{1,0}^{p-1}(y) \left[y \cdot \nabla U_{1,0}(y) + {2U_{1,0}(y) \over p-1}\right] \({\pa U_{1,0} \over \pa \lambda}\)(y) dy = - \lambda_i^{n-2 \over 2} d_i C_4
\end{aligned}
\end{equation}
where $C_4 = \int_{\mr^n} U_{1,0}^{p-1} \({\pa U_{1,0} \over \pa \lambda}\)^2 > 0$.
Consequently, \eqref{eq-weak-conv-1}, \eqref{eq-weak-conv-2} and \eqref{eq-w-v} enable us to deduce that
\begin{equation}\label{eq-weak-conv-4}
\ma_3 \hat{\mathbf{d}}_{\ell}^1
= {2pC_4 \over (n-2)C_2C_3} \lim_{\ep \to 0} \({\mu_{\ell \ep}-1 \over \ep}\) \hat{\mathbf{d}}_{\ell}^1 \quad \text{where } \hat{\mathbf{d}}_{\ell}^1 = \(\begin{array}{c}
\lambda_1^{n-2 \over 2} d_{\ell, 1} \\ \cdots \\ \lambda_m^{n-2 \over 2} d_{\ell, m}
\end{array}\) \ne 0.
\end{equation}
Multiplying a row vector $\(\hat{\mathbf{d}}_{\ell}^1\)^T$ in the both sides yields
\begin{equation}\label{eq-weak-conv-3}
\lim_{\ep \to 0} \({\mu_{\ell \ep}-1 \over \ep}\) = {(n-2)^2C_2C_3 \over 2(n+2)C_4} \cdot \({\(\hat{\mathbf{d}}_{\ell}^2\)^T \mm_2 \hat{\mathbf{d}}_{\ell}^2 \over \left|\hat{\mathbf{d}}_{\ell}^1\right|^2} + C_0\)
\end{equation}
where $\hat{\mathbf{d}}_{\ell}^2 = \(\lambda_1^{n-2} d_{\ell, 1}, \cdots, \lambda_m^{n-2} d_{\ell, m}\)^T$ and $\mm_2$ is the matrix introduced in Lemma \ref{lem-pos}.
However the right-hand side of \eqref{eq-weak-conv-3} is positive due to Lemma \ref{lem-pos}, and this contradicts the bound of $\mu_{\ell \ep}$ provided in Proposition \ref{prop-apriori}.
Hence it should hold that $d_{\ell, i} =0$ for all $i$.
The proof is finished.
\end{proof}

This result improves our knowledge on the limit behavior of the $\ell$-th eigenvalues (see Corollary \ref{cor-apriori}) for $m+1 \le \ell \le (n+1)m$, which is essential in the next section.
\begin{cor}\label{cor-apriori-2}
For $m+1 \le \ell \le (n+1)m$, one has
\begin{equation}\label{eq-apriori-b}
\left|\mu_{\ell \ep} - 1\right| = O\(\ep^{n-1 \over n-2}\) \quad \text{as } \ep \to 0.
\end{equation}
\end{cor}
\begin{proof}
By Proposition \ref{prop-weak-conv} and Lemma \ref{lem-weak-conv} (1), there is $i_1 \in \{1, \cdots, m\}$ such that
\[\tilde{v}_{\ell i_1 \ep} \rightharpoonup \sum_{k=1}^n d_{\ell,i_1,k} \(\frac{\pa U_{1,0}}{\pa (x_0)_k}\) \quad \text{weakly in } H^1(\mr^n)\]
where $(d_{\ell,i_1,1},\cdots, d_{\ell,i_1,n}) \ne 0$.
Without any loss of generality, we may assume that $d_{\ell,i_1,1} \ne 0$.
By differentiating the both sides of \eqref{eq-main-e}, we get
\begin{equation}\label{eq-apriori-a}
-\Delta \frac{\pa u_{\ep}}{\pa x_1} = (p-\ep) u_{\ep}^{p-1-\ep} \frac{\pa u_{\ep}}{\pa x_1}.
\end{equation}
Let us multiply \eqref{eq-apriori-a} by $v_{\ell \ep}$ and \eqref{eq-lin} by ${\pa u_{\ep} \over \pa x_1}$, respectively,
integrate both of them over $B^n(x_{i_1 \ep},r)$ for a small fixed $r > 0$ and subtract the first equation from the second to derive
\begin{equation}\label{eq-apriori-c}
\int_{\pa B^n(x_{i_1 \ep},r)} \left\{ \frac{\pa}{\pa \nu}\(\frac{\pa u_{\ep}}{\pa x_1}\) v_{\ell\ep}-\frac{\pa u_{\ep}}{\pa x_1}\frac{\pa v_{\ell\ep}}{\pa \nu}\right\} dS
= (p-\ep)\(\mu_{\ell \ep} - 1\) \int_{B^n(x_{i_1 \ep},r)} u_{\ep}^{p-1-\ep} \frac{\pa u_{\ep}}{\pa x_1} v_{\ell\ep}.
\end{equation}
By Lemma \ref{lem-u-d-bound}, its left-hand side is $O\(\ep^{3/2}\)$ while the right-hand side is computed as
\begin{equation}\label{eq-apriori-d}
\begin{aligned}
\int_{B^n(x_{i_1 \ep},r)} u_{\ep}^{p-1-\ep} \frac{\pa u_{\ep}}{\pa x_1} v_{\ell\ep}
&= \(\lambda_{i_1} \ep^{\alpha_0}\)^{n-(\sigma_\ep+1)-2} \int_{B^n(0,\(\lambda_i\ep^{\alpha_0})^{-1}r\)} \tu_{i_1 \ep}^{p-1-\ep} \frac{\pa \tu_{i_1 \ep}}{\pa x_1} \tilde{v}_{\ell i_1 \ep} \\
&= - \lambda_{i_1}^{n-4 \over 2} \ep^{n-4 \over 2(n-2)} \(d_{\ell, i_1, 1} \int_{\mr^n} U_{1,0}^{p-1}\({\pa U_{1,0} \over \pa x_1}\)^2 + o(1)\).
\end{aligned}
\end{equation}
Therefore, if we denote $C_5 = \int_{\mr^n} U_{1,0}^{p-1}\({\pa U_{1,0} \over \pa x_1}\)^2 > 0$, we deduce that
\[O\(\ep^{3 \over 2}\) = - \lambda_{i_1}^{n-4 \over 2} \ep^{n-4 \over 2(n-2)} (p+o(1)) \left[\lim_{\ep \to 0}(\mu_{\ell \ep}-1)\right] \(d_{\ell, i, 1} C_5 + o(1)\),\]
which leads the desired estimate \eqref{eq-apriori-b}.
\end{proof}

\section{A further analysis on asymptotic behavior of the $\ell$-th eigenfunctions, $m+1 \le \ell \le (n+1)m$}\label{sec-5}
In view of Lemma \ref{lem-weak-conv} and the proof of Proposition \ref{prop-weak-conv},
we know that $\ep^{-1}v_{\ell \ep} \to 0$ as $\ep \to 0$ uniformly in $\Omega$ outside of the blow-up points $\{x_1, \cdots, x_m\}$.
Motivated by the argument in \cite{GGOS}, we prove its improvement \eqref{eq-eigenf-first-2} here, which is stated once more in the following proposition.
\begin{prop}\label{prop-opt-conv}
Let $\mm_1$ and $\map$ be the matrices defined in \eqref{eq-mtx-M_1} and \eqref{eq-mtx-P}, respectively.
Also we remind a column vector $\mathbf{d}_{\ell} \in \mr^{mn}$ in \eqref{eq-eigenv-first-2} and set two row vectors $\mg(x)$ and $\widetilde{\mg}(x)$ by
\begin{equation}\label{eq-mtx-G}
\mg(x) = \(G(x,x_1), \cdots, G(x,x_m)\) \in \mr^m, \quad \widetilde{\mg}(x) = \(\lambda_1^{n \over 2}\nabla_yG(x,x_1), \cdots, \lambda_m^{n \over 2}\nabla_yG(x,x_m)\) \in \mr^{mn}
\end{equation}
for any $x \in \Omega$. If $m+1 \le \ell \le (n+1)m$, then
\begin{equation}\label{eq-opt-conv}
\ep^{-{n-1 \over n-2}} v_{\ell \ep}(x) \to C_1 \(\mg(x)\mm_1^{-1}\map + \widetilde{\mg}(x)\) \mathbf{d}_{\ell},
\end{equation}
in $C^1\(\Omega \setminus \{x_1, \cdots, x_m\}\)$ as $\ep \to 0$ where $C_1 > 0$ is a constant in Theorem \ref{thm-eigen-first}.
\end{prop}
\begin{rem}
If we write \eqref{eq-opt-conv} in terms of the components of the vectors $\mg(x)$ and $\widetilde{\mg}(x)$, and matrices $\mm_1^{-1}$ and $\map$, we get \eqref{eq-eigenf-first-2}.
\end{rem}
\noindent We will present the proof by dividing it into several lemmas.
The first lemma is a variant of Lemmas \ref{lem-u-asym} and \ref{lem-weak-conv} (2).
\begin{lem}\label{lem-uv-asym}
Given a small fixed number $r > 0$, it holds that
\begin{align}
u_{\ep}(x) &= \sum_{i=1}^m \kappa_{i0} G(x,x_{i\ep}) + o\(\ep^{n \over 2(n-2)}\) \nonumber
\intertext{and}
{v_{\ell \ep}(x) \over \mu_{\ell \ep}(p-\ep)} &= \sum_{i=1}^m \(\kappa_{i1} G(x,x_{i\ep}) + \bk_{i2} \cdot \nabla_y G(x,x_{i\ep})\) + o\(\ep^{n-1 \over n-2}\) \label{eq-v-conv-2}
\end{align}
in $C^1\(\Omega \setminus \{x_1,\cdots,x_m\}\)$ as $\ep \to 0$ where
\[\kappa_{i0} = \int_{B^n(x_{i\ep},r)} u_{\ep}^{p-\ep} = O\(\sqrt{\ep}\), \quad \kappa_{i1} = \int_{B^n(x_{i\ep},r)} u_{\ep}^{p-1-\ep}v_{\ell \ep} = O(\ep)\]
and $\bk_{i2} = (\kappa_{i21}, \cdots, \kappa_{i2n}) \in \mr^n$ is a row vector such that
\begin{equation}\label{eq-k-i2}
\bk_{i2} = \int_{B^n(x_{i\ep},r)} (y-x_{i\ep}) \(u_{\ep}^{p-1-\ep}v_{\ell \ep}\)(y) dy = O\(\ep^{n-1 \over n-2}\)
\end{equation}
(note that $\kappa_{i0}$, $\kappa_{i1}$ and $\bk_{i2}$ depend also on $\ep$ or $\ell$).
\end{lem}
\begin{proof}
The proof is similar to Lemmas \ref{lem-u-asym} and \ref{lem-weak-conv} (2), so we just briefly sketch why \eqref{eq-v-conv-2} holds in $C^0(K)$ for any compact subset $K$ of $\Omega \setminus \{x_1,\cdots,x_m\}$.
For $x \in A_r$ (see \eqref{eq-A_r}), a combination of Green's representation formula and the Taylor expansion of $G(x,y)$ in the $y$-variable show that
\begin{multline*}
{v_{\ell \ep}(x) \over \mu_{\ell \ep}(p-\ep)}\\
= \sum_{i=1}^m \int_{B^n(x_{i \ep}, r/2)} \(G(x,x_{i\ep}) + (y-x_{i\ep}) \cdot \nabla_y G(x,x_{i\ep}) + O\(|y-x_{i\ep}|^2\)\) \(u_{\ep}^{p-1-\ep} v_{\ell \ep}\)(y)dy + O\(\ep^{n \over n-2}\)
\end{multline*}
Also, by means of Proposition \ref{prop-LZ} and Lemma \ref{lem-w-bound}, we have
\begin{align*}
\int_{B^n(x_{i \ep}, r/2)} |y-x_{i\ep}|^2 \cdot \left|\(u_{\ep}^{p-1-\ep} v_{\ell \ep}\)(y)\right| dy
&= \(\lambda_{i\ep}\ep^{\alpha_0}\)^n \int_{B^n\(0, \(\lambda_{i\ep}\ep^{\alpha_0}\)^{-1}r/2\)} |x|^2 \cdot \left|\(\tu_{\ep}^{p-1-\ep} \tilde{v}_{\ep}\)(x)\right|dx \\
&\le C \ep^{n \over n-2} \int_0^{C\ep^{-{1 \over n-2}}} {t^{n+1} \over 1+t^{(n+2)-(n-2)\ep}}dt = O\(\ep^{n \over n-2}\)
\end{align*}
for each $i$, from which the desired result follows.
The order of $k_{i0}, k_{i1}$ and $\bk_{i2}$ can be computed as in \eqref{eq-2-3} or \eqref{eq-weak-conv-a}.
\end{proof}

\noindent Let us write $u_{\ep}$ and $v_{\ell \ep}$ in the following way. For each $i = 1, \cdots, m$,
\begin{equation}\label{eq-g-ie}
u_{\ep}(x) = {\kappa_{i0}\gamma_n \over |x-x_{i\ep}|^{n-2}} + g_{i\ep}(x) + o\(\ep^{n \over 2(n-2)}\)
\quad \text{where} \quad g_{i\ep}(x) = -\kappa_{i0}H(x,x_{i\ep}) + \sum_{j \ne i} \kappa_{j0}G(x,x_{j\ep}),
\end{equation}
and
\begin{equation}\label{eq-h-ie-0}
{v_{\ell \ep}(x) \over \mu_{\ell \ep}(p-\ep)} = {\kappa_{i1}\gamma_n \over |x-x_{i\ep}|^{n-2}} + (n-2)\gamma_n \bk_{i2} \cdot {x-x_{i\ep} \over |x-x_{i\ep}|^n} + h_{i\ep}(x) + o\(\ep^{n-1 \over n-2}\)
\end{equation}
where
\begin{equation}\label{eq-h-ie-1}
h_{i\ep}(x) = - \(\kappa_{i1} H(x,x_{i\ep}) + \bk_{i2} \cdot \nabla_y H(x,x_{i\ep})\) + \sum_{j \ne i} \(\kappa_{j1} G(x,x_{j\ep}) + \bk_{j2} \cdot \nabla_y G(x,x_{j\ep})\).
\end{equation}
Note that $g_{i\ep}$ an $h_{i\ep}$ are harmonic in a neighborhood of $x_{i\ep}$.
With these decompositions we now compute $\kappa_{i1}$, will be shown to be $O\(\ep^{n-1 \over n-2}\)$, by applying the bilinear version of the Poho\v{z}aev identity which the next lemma describes.
\begin{lem}
For any point $x_0 \in \mr^n$, a positive number $r > 0$ and functions $f,\ g \in C^2\(\overline{B^n(x_0,r)}\)$, it holds that
\begin{multline}\label{eq-poho}
\int_{B^n(x_0,r)} \left[\((x-x_0) \cdot \nabla f\) \Delta g + \((x-x_0) \cdot \nabla g\)\Delta f\right] \\
= r \int_{\pa B^n(x_0,r)} \(2 \frac{\pa f}{\pa \nu} \frac{\pa g}{\pa \nu} - \nabla f \cdot \nabla g\) + (n-2) \int_{B^n(x_0,r)} \nabla f \cdot \nabla g
\end{multline}
where $\nu$ is the outward unit normal vector on $\pa B^n(x_0,r)$.
\end{lem}
\begin{proof}
This follows from an elementary computation. See the proof of \cite[Proposition 5.5]{Ot} in which the author considered it when $n = 2$.
\end{proof}

\begin{lem}\label{lem-est-k_1i}
Recall the definition of $\mm_1$ in \eqref{eq-mtx-M_1} and its inverse $\mm_1^{-1} = \(m_1^{ij}\)_{1 \le i,j \le m}$.
Then it holds for $m+1 \le \ell \le (n+1)m$ that
\begin{equation}\label{eq-est-k_1i-0}
\ep^{-{n-1 \over n-2}} \kappa_{i1} = \sum_{j=1}^m m_1^{ij}\(-{1 \over 2} \ep^{-{n-1 \over n-2}} \bk_{j2} \cdot \nabla \tau(x_j) + \sum_{l \ne j} \ep^{-{n-1 \over n-2}} \bk_{l2} \cdot \nabla_y G(x_j,x_l)\) + o(1).
\end{equation}
\end{lem}
\begin{rem}\label{rem-est-k_1i}
If $m = 1$, one has that $\Upsilon_1(\lambda_1, x_1) = c_1 \tau_1(x_1)\lambda_1^{n-2} - c_2 \log \lambda_1$ (refer to \eqref{eq-upsilon}).
Therefore \eqref{eq-est-k_1i-0} and $0 = \pa_{x_1} \Upsilon_1(\lambda_1,x_1) = c_1 \(\pa_{x_1}\tau\)(x_1) \lambda_1^{n-2}$ imply $\ep^{-{n-1 \over n-2}} \kappa_{i1} = o(1)$.
\end{rem}
\begin{proof}
Fixing a sufficiently small number $r > 0$, we take $x_0 = x_{i\ep}$, $f = u_{\ep}$ and $g = v_{\ell \ep}$ for \eqref{eq-poho}.
Then from \eqref{eq-main-e}, \eqref{eq-lin} and the estimate
\begin{align*}
&\ (1-\mu_{\ell \ep}) \int_{B^n(x_{i\ep},r)} \left[(x-x_{i\ep}) \cdot \nabla u_{\ep} \right] u_{\ep}^{p-1-\ep} v_{\ell \ep} \\
&= O\(\ep^{n-1 \over n-2}\) \cdot \ep^{1 \over 2} \lambda_i^{n-2 \over 2} \(-\sum_{k=1}^n d_{\ell, i, k} \int_{\mr^n} \(x \cdot \nabla U_{1,0}\) U_{1,0}^{p-1} {\pa U_{1,0} \over \pa x_k} + o(1)\)
= o\(\ep^{{n-1 \over n-2}+{1 \over 2}}\)
\end{align*}
where Proposition \ref{prop-weak-conv} and Corollary \ref{cor-apriori-2} are made use of, one finds that the left-hand side of \eqref{eq-poho} is equal to
\begin{align*}
&\ - \int_{B^n(x_{i\ep},r)} (x-x_{i\ep}) \cdot \nabla\(u_{\ep}^{p-\ep}v_{\ell \ep}\)
+ (1-\mu_{\ell \ep})(p-\ep) \int_{B^n(x_{i\ep},r)} \left[(x-x_{i\ep}) \cdot \nabla u_{\ep} \right] u_{\ep}^{p-1-\ep} v_{\ell \ep}\\
&= n \int_{B^n(x_{i\ep},r)} u_{\ep}^{p-\ep}v_{\ell \ep} + o\(\ep^{{n-1 \over n-2}+{1 \over 2}}\). 
\end{align*}
As a result, \eqref{eq-poho} reads as
\begin{equation}\label{eq-est-k_1i-1}
\begin{aligned}
&\ r \int_{\pa B^n(x_{i\ep},r)} \(2 \frac{\pa u_{\ep}}{\pa \nu} \frac{\pa v_{\ell \ep}}{\pa \nu} - \nabla u_{\ep} \cdot \nabla v_{\ell \ep}\)
+ (n-2) \int_{\pa B^n(x_{i\ep},r)} {\pa u_{\ep} \over \pa \nu} v_{\ell \ep} \\ 
&= 2 \int_{B^n(x_{i\ep},r)} u_{\ep}^{p-\ep} v_{\ell \ep}
+ o\(\ep^{{n-1 \over n-2}+{1 \over 2}}\) \\
&= 2\left[\mu_{\ell \ep}(p-\ep)-1\right]^{-1} \int_{\pa B^n(x_{i\ep},r)} \({\pa u_{\ep} \over \pa \nu} v_{\ell \ep} - {\pa v_{\ell \ep} \over \pa \nu} u_{\ep} \) dS + o\(\ep^{{n-1 \over n-2}+{1 \over 2}}\)
\end{aligned}
\end{equation}
where the latter equality is due to Green's identity \eqref{eq-green-id}.

We compute the rightmost side of \eqref{eq-est-k_1i-1} first.
Since $g_{i\ep}$, $h_{i\ep}$ and $(x-x_{i\ep})\cdot \nabla g_{i\ep}$ are harmonic near $x_{i\ep}$ (see \eqref{eq-g-ie} and \eqref{eq-h-ie-1} to remind their definitions), a direct computation with \eqref{eq-g-ie}-\eqref{eq-h-ie-1}, the mean value formula and Green's identity \eqref{eq-green-id} shows that
\begin{equation}\label{eq-est-k_1i-2}
\begin{aligned}
&\ \int_{\pa B^n(x_{i\ep},r)} \({\pa u_{\ep} \over \pa \nu} v_{\ell \ep} - {\pa v_{\ell \ep} \over \pa \nu} u_{\ep} \) dS \\
&= \mu_{\ell\ep}(p-\ep) \left[ (n-2) \gamma_n \left|S^{n-1}\right| \(\kappa_{i1} g_{i\ep}(x_{i\ep}) - \kappa_{i0}h_{i\ep}(x_{i\ep})\)
+ {(n-2)\gamma_n \over r^n} \bk_{i2} \cdot \int_{\pa B^n(x_{i\ep},r)} (x-x_{i\ep}){\pa g_{i\ep} \over \pa \nu} dS \right.\\
&\hspace{60pt} \left. + {(n-2)(n-1)\gamma_n \over r^{n+1}} \bk_{i2} \cdot \int_{\pa B^n(x_{i\ep},r)} (x-x_{i\ep})g_{i\ep} dS + o\(\ep^{{n-1 \over n-2}+{1 \over 2}}\) \right].
\end{aligned}
\end{equation}
Moreover, both $g_{i\ep}$ and ${x-x_{i\ep} \over |x-x_{i\ep}|^n}$ are harmonic in $B^n(x_{i\ep},r) \setminus \{x_{i\ep}\}$,
so Green's identity again infers that the value
\begin{equation}\label{eq-est-k_1i-22}
\begin{aligned}
I_{1r} &:= \bk_{i2} \cdot \int_{\pa B^n(x_{i\ep},r)} \({x-x_{i\ep} \over |x-x_{i\ep}|^n} {\pa g_{i\ep} \over \pa \nu} + (n-1) {x-x_{i\ep} \over |x-x_{i\ep}|^{n+1}} g_{i\ep}\) dS \\
& = \bk_{i2} \cdot \int_{\pa B^n(x_{i\ep},r)} \left[{x-x_{i\ep} \over |x-x_{i\ep}|^n} {\pa g_{i\ep} \over \pa \nu} - {\pa \over \pa \nu}\({x-x_{i\ep} \over |x-x_{i\ep}|^n}\) g_{i\ep}\right] dS
\end{aligned}
\end{equation}
is independent of $r > 0$.
Thus, taking the limit $r \to 0$ and applying the Taylor expansion of $g_{i\ep}$, we find that it is equal to
\begin{equation}\label{eq-est-k_1i-23}
\begin{aligned}
I_{10} &:= \lim_{r \to 0} I_{1r} \\
&= \lim_{r \to 0} \sum_{k,l=1}^n {\kappa_{i2k} \over r^{n+1}} \int_{\pa B^n(0,r)} x_k x_l \left[ \(\pa_l g_{i\ep}\)(x_{i\ep}) + O(|x|) \right] dS \\
&\quad + (n-1) \lim_{r \to 0} \sum_{k=1}^n {\kappa_{i2k} \over r^{n+1}} \int_{\pa B^n(0,r)} x_k \left[ g_{i\ep}(x_{i\ep}) + \sum_{l=1}^n x_l \(\pa_l g_{i\ep}\)(x_{i\ep}) + O\(|x|^2\) \right] dS \\
&= n \sum_{k,l=1}^n \kappa_{i2k}\(\pa_l g_{i\ep}\)(x_{i\ep}) \int_{\pa B^n(0,1)} x_k x_l dS = \left|S^{n-1}\right| \bk_{i2} \cdot \nabla g_{i\ep}(x_{i\ep}).
\end{aligned}
\end{equation}
However the quantity $\bk_{i2} \cdot \nabla g_{i\ep}(x_{i\ep})$ is negligible in the sense that its order is $\ep^{{n-1 \over n-2}+{1 \over 2}}$,
because $\bk_{i2} = O\(\ep^{n-1 \over n-2}\)$ and that $\nabla_x \Upsilon_m (\lambda_1, \cdots, \lambda_m, x_1, \cdots, x_m) = 0$ means
\begin{equation}\label{eq-est-k_1i-24}
\begin{aligned}
\lim_{\ep \to 0} \ep^{-{1 \over 2}} \nabla g_{i\ep}(x_{i\ep}) &= - \lim_{\ep \to 0} \(\ep^{-{1 \over 2}} \kappa_{i0}\) \(\nabla_x H\)(x_{i\ep}, x_{i\ep}) + \sum_{j \ne i} \lim_{\ep \to 0}\(\ep^{-{1 \over 2}} \kappa_{j0}\) \(\nabla_x G\)(x_{i\ep}, x_{j\ep}) \\
&= \(- {1 \over 2} \lambda_i^{n-2 \over 2} \(\nabla_x \tau\)(x_i) + \sum_{j \ne i} \lambda_j^{n-2 \over 2} \(\nabla_x G\)(x_i, x_j)\) C_2 = 0
\end{aligned}
\end{equation}
where $C_2 = \int_{\mr^n} U_{1,0}^p$ as before. Hence we can conclude that
\begin{equation}\label{eq-est-k_1i-21}
I_{10} = o\(\ep^{{n-1 \over n-2}+{1 \over 2}}\).
\end{equation}

Regarding the leftmost side of \eqref{eq-est-k_1i-1}, one gets in a similar fashion to the derivation of \eqref{eq-est-k_1i-2} that
\begin{equation}\label{eq-est-k_1i-3}
\begin{aligned}
&\ \int_{\pa B^n(x_{i\ep},r)} \frac{\pa u_{\ep}}{\pa \nu} \frac{\pa v_{\ell \ep}}{\pa \nu} dS \\
&= \mu_{\ell\ep}(p-\ep) \left[{(n-2)^2 \gamma_n^2 \left|S^{n-1}\right| \kappa_{i0}\kappa_{i1} \over r^{n-1}}
- {(n-2)(n-1)\gamma_n \over r^{n+1}} \bk_{i2} \cdot \int_{\pa B^n(x_{i\ep},r)} (x-x_{i\ep}){\pa g_{i\ep} \over \pa \nu} dS \right.
\\
&\hspace{60pt} \left. + \int_{\pa B^n(x_{i\ep},r)}{\pa g_{i\ep} \over \pa \nu}{\pa h_{i\ep} \over \pa \nu} dS
+ o\(\ep^{{n-1 \over n-2}+{1 \over 2}}\) \right].
\end{aligned}
\end{equation}
Furthermore, we have
\begin{equation}\label{eq-est-k_1i-4}
\begin{aligned}
&\ \int_{\pa B^n(x_{i\ep},r)} \nabla u_{\ep} \cdot \nabla v_{\ell \ep} dS \\
&= \mu_{\ell\ep}(p-\ep) \left[{(n-2)^2 \gamma_n^2 \left|S^{n-1}\right| \kappa_{i0}\kappa_{i1} \over r^{n-1}}
- {n(n-2)\gamma_n \over r^{n+1}} \bk_{i2} \cdot \int_{\pa B^n(x_{i\ep},r)} (x-x_{i\ep}) {\pa g_{i\ep} \over \pa \nu} dS \right.\\
&\hspace{60pt} \left. + {(n-2)\gamma_n \over r^n} \bk_{i2} \cdot \int_{\pa B^n(x_{i\ep},r)} \nabla g_{i\ep} dS + \int_{\pa B^n(x_{i\ep},r)} \nabla g_{i\ep} \cdot \nabla h_{i\ep} dS + o\(\ep^{{n-1 \over n-2}+{1 \over 2}}\) \right].
\end{aligned}
\end{equation}
and
\begin{equation}\label{eq-est-k_1i-5}
\begin{aligned}
&\ \int_{\pa B^n(x_{i\ep},r)} {\pa u_{\ep} \over \pa \nu} v_{\ell \ep} dS \\
&= \mu_{\ell\ep}(p-\ep) \left[-{(n-2) \gamma_n^2  \left|S^{n-1}\right| \kappa_{i0}\kappa_{i1} \over r^{n-2}} - (n-2)\gamma_n \left|S^{n-1}\right| \kappa_{i0} h_{i\ep}(x_{i\ep})
\right. \\
&\hspace{30pt} \left. + {(n-2)\gamma_n \over r^n} \bk_{i2} \cdot \int_{\pa B^n(x_{i\ep},r)} (x-x_{i\ep}){\pa g_{i\ep} \over \pa \nu} dS
+ \int_{\pa B^n(x_{i\ep},r)} {\pa g_{i\ep} \over \pa \nu} h_{i\ep} dS + o\(\ep^{{n-1 \over n-2}+{1 \over 2}}\) \right].
\end{aligned}
\end{equation}

Therefore putting \eqref{eq-est-k_1i-2} and \eqref{eq-est-k_1i-21}-\eqref{eq-est-k_1i-5} into \eqref{eq-est-k_1i-1} gives that
\begin{equation}\label{eq-est-k_1i-6}
\begin{aligned}
&\(\mu_{\ell \ep}(p-\ep)-1\) \left[2r \int_{\pa B^n(x_{i\ep},r)}{\pa g_{i\ep} \over \pa \nu}{\pa h_{i\ep} \over \pa \nu} dS
- {(n-2)\gamma_n \over r^{n-1}} \bk_{i2} \cdot \int_{\pa B^n(x_{i\ep},r)} \nabla g_{i\ep} dS \right.\\
&\ \left. - r \int_{\pa B^n(x_{i\ep},r)} \nabla g_{i\ep} \cdot \nabla h_{i\ep} dS - (n-2)^2 \gamma_n \left|S^{n-1}\right| \kappa_{i0} h_{i\ep}(x_{i\ep})
+ (n-2) \int_{\pa B^n(x_{i\ep},r)} {\pa g_{i\ep} \over \pa \nu} h_{i\ep} dS \right] \\
& = 2 \left[ (n-2) \gamma_n \left|S^{n-1}\right| \(\kappa_{i1} g_{i\ep}(x_{i\ep}) - \kappa_{i0}h_{i\ep}(x_{i\ep})\) + o\(\ep^{{n-1 \over n-2}+{1 \over 2}}\) \right].
\end{aligned}
\end{equation}
Noticing that each component of $\nabla g_{i\ep}$ is harmonic, we obtain
\[ {1 \over r^{n-1}} \bk_{i2} \cdot \int_{\pa B^n(x_{i\ep},r)} \nabla g_{i\ep} dS = \left|S^{n-1}\right| \bk_{i2} \cdot \nabla g_{i\ep} (x_{i\ep}) = o\(\ep^{{n-1 \over n-2}+{1 \over 2}}\),\]
where the second equality was deduced in \eqref{eq-est-k_1i-24}.
Also, by setting $f = g_{i\ep}$, $g = h_{i\ep}$ and $x_0 = x_{i\ep}$ in the bilinear Poho\v{z}aev identity \eqref{eq-poho}, one can verify that
\[r\(\int_{\pa B^n(x_{i\ep},r)} 2 {\pa g_{i\ep} \over \pa \nu}{\pa h_{i\ep} \over \pa \nu} - \nabla g_{i\ep} \cdot \nabla h_{i\ep}\) dS + (n-2) \int_{\pa B^n(x_{i\ep},r)} {\pa g_{i\ep} \over \pa \nu} h_{i\ep} dS = 0.\]
Subsequently, \eqref{eq-est-k_1i-6} is reduced to
\[2 \kappa_{i1} \(\ep^{-{1 \over 2}} g_{i\ep}(x_{i\ep})\) = \left[ 2-\(\mu_{\ell \ep}(p-\ep)-1\)(n-2) \right] \(\ep^{-{1 \over 2}}\kappa_{i0}\) h_{i\ep}(x_{i\ep}) + o\(\ep^{n-1 \over n-2}\).\]

Now we employ $\nabla_{\lambda} \Upsilon_m (\lambda_1, \cdots, \lambda_m, x_1, \cdots, x_m) = 0$ to see that
\[\ep^{-{1 \over 2}} g_{i\ep}(x_{i\ep}) = C_2 \left[-\tau(x_i) \lambda_i^{n-2 \over 2} + \sum_{j \ne i} G(x_i,x_j) \lambda_j^{n-2 \over 2}\right] + o(1) = - {C_2 c_2 \over c_1(n-2)\lambda_i^{n-2 \over 2}} + o(1)\]
and that $\ep^{-{1 \over 2}}\kappa_{i0} = \lambda_i^{n-2 \over 2}C_2 + o(1)$,
where $C_2 > 0$ is the constant that appeared in \eqref{eq-est-k_1i-24} and $c_1, c_2 > 0$ are the numbers in \eqref{eq-upsilon-2}.
Consequently, we have
\begin{align*}
&\ \(C_0\lambda_i^{-(n-2)} + o(1)\) \kappa_{i1} = h_{i\ep}(x_{i\ep}) + o\(\ep^{n-1 \over n-2}\) \\
&= - \left[\kappa_{i1} \tau(x_{i\ep}) + {1 \over 2} \bk_{i2} \cdot \nabla \tau(x_{i\ep})\right] + \sum_{j \ne i} \(\kappa_{j1} G(x_{i\ep},x_{j\ep}) + \bk_{j2} \cdot \nabla_y G(x_{i\ep},x_{j\ep})\) + o\(\ep^{n-1 \over n-2}\),
\end{align*}
which can be rewritten as
\[\(\mm_1 + o(1)\) \(\begin{array}{c}
\kappa_{11} \\ \vdots \\ \kappa_{m1}
\end{array}\) = \(\begin{array}{c}
- \dfrac{1}{2} \bk_{12} \cdot \nabla \tau(x_1) + \sum\limits_{j \ne 1} \bk_{j2} \cdot \nabla_y G(x_1, x_j) \\ \vdots \\ - \dfrac{1}{2} \bk_{m2} \cdot \nabla \tau(x_m) + \sum\limits_{j \ne m} \bk_{j2} \cdot \nabla_y G(x_m, x_j)
\end{array}\) + o\(\ep^{n-1 \over n-2}\).\]
This is nothing but \eqref{eq-est-k_1i-0}.
\end{proof}

\begin{proof}[Proof of Proposition \ref{prop-opt-conv}]
According to \eqref{eq-k-i2} and Proposition \ref{prop-weak-conv}, we have
\begin{align*}
\ep^{-{n-1 \over n-2}} \kappa_{i2k} &= \ep^{-{n-1 \over n-2}} \int_{B^n(x_{i\ep},r)} (y-x_{i\ep})_k \(u_{\ep}^{p-1-\ep}v_{\ell \ep}\)(y) dy = \lambda_i^{n-1} d_{\ell, i, k} \(- \int_{\mr^n} x_1 U_{1,0}^{p-1} {\pa U_{1,0} \over \pa x_1}\) + o(1)\\
&= \lambda_i^{n-1} d_{\ell, i, k} p^{-1}C_1 + o(1)
\end{align*}
for any $i \in \{1, \cdots, m\}$ and $k \in \{1, \cdots, n\}$.
Hence the proposition follows from \eqref{eq-v-conv-2}, Corollary \ref{cor-apriori} (or Corollary \ref{cor-apriori-2}) and Lemma \ref{lem-est-k_1i}.
\end{proof}

\section{Characterization of the $\ell$-th eigenvalues, $m+1 \le \ell \le (n+1)m$}\label{sec-6}
Our goal in this section is to perform the proof of Theorem \ref{thm-eigen-first-2}.
For the convenience, we restate it in the following proposition.
\begin{prop}\label{prop-eigen-cha}
Let $\ma_2$ be the matrix which was introduced in the statement of Theorem \ref{thm-eigen-first-2} and $\rho^2_{\ell}$ the $(\ell-m)$-th eigenvalue of $\ma_2$.
For $m+1 \le \ell \le (n+1)m$, the $\ell$-th eigenvalue $\mu_{\ell \ep}$ for linear problem \eqref{eq-lin} satisfies that
\begin{equation}\label{eq-eigen-cha-0}
\mu_{\ell \ep} = 1 - c_0 \rho^2_{\ell} \ep^{n \over n-2} + o\(\ep^{n \over n-2}\) \quad \text{where} \quad c_0 = (C_1C_2)/(pC_5) > 0.
\end{equation}
In addition, the nonzero vector $\mathbf{d}_{\ell} \in \mr^{mn}$ defined via \eqref{eq-eigenv-first-2} is an eigenfunction of $\ma_2$ corresponding to $\rho^2_{\ell}$
and satisfies $\mathbf{d}_{\ell_1}^T \cdot \mathbf{d}_{\ell_2}^T = 0$ if $m+1 \le \ell_1 \ne \ell_2 \le (n+1)m$.
\end{prop}
\noindent The next lemma contains a key computation for the proof of Proposition \ref{prop-eigen-cha}.
\begin{lem}\label{lem-jjl}
Define
\begin{align}
\mj_{jl;ik}^r = \mj_{jl}^r &= \int_{\pa B^n(x_i,r)} \left[{\pa \over \pa \nu_x}\({\pa G \over \pa x_k}(x,x_j)\)G(x,x_l) - {\pa G \over \pa x_k}(x,x_j) {\pa G \over \pa \nu_x}(x,x_l)\right] \label{eq-r-2}
\intertext{and}
\mk_{jl;ikq}^r = \mk_{jl}^r &= \int_{\pa B^n(x_i,r)} \left[{\pa \over \pa \nu_x}\({\pa G \over \pa x_k}(x,x_j)\) {\pa G \over \pa y_q}(x,x_l)
- {\pa G \over \pa x_k}(x,x_j) {\pa \over \pa \nu_x}\({\pa G \over \pa y_q}(x,x_l)\)\right] \label{eq-r-3}
\end{align}
for each $i,\ j,\ l \in \{1, \cdots, m\}$ and $k,\ q \in \{1, \cdots, n\}$,
where the outward unit normal derivative ${\pa \over \pa \nu_x}$ acts over the $x$-variable of Green's function $G = G(x,y)$.
Then they are the value independent of $r > 0$ and calculated as
\[\mj_{jl}^r = \begin{cases}
0 &\text{if } j \ne i \text{ and } l \ne i, \\
\dfrac{\pa G}{\pa x_k}(x_i, x_l) &\text{if } j = i \text{ and } l \ne i,\\
\dfrac{\pa G}{\pa x_k}(x_i, x_j) &\text{if } j \ne i \text{ and } l = i,\\
-\dfrac{\pa \tau}{\pa x_k}(x_i) &\text{if } j = l = i,
\end{cases} \quad \text{and} \quad
\mk_{jl}^r = \begin{cases}
0 &\text{if } j \ne i \text{ and } l \ne i, \\
\dfrac{\pa^2 G}{\pa x_k \pa y_q}(x_i,x_l) &\text{if } j = i \text{ and } l \ne i,\\
\dfrac{\pa^2 G}{\pa x_k \pa x_q}(x_i,x_j) &\text{if } j \ne i \text{ and } l = i,\\
-\dfrac{1}{2} \dfrac{\pa^2 \tau}{\pa x_k \pa x_q}(x_i) &\text{if } j = l = i.
\end{cases}\]
\end{lem}
\begin{proof}
As explained in the proof of Lemma \ref{lem-ijl}, the integral $\mj_{jl}^r$ in \eqref{eq-r-2} is independent of $r > 0$, so one may take $r \to 0$ to find its value.
We compute each $\mj_{jl}^r$ by considering four mutually exclusive cases categorized according to the relation of indices $j$, $l$ and $i$.

\medskip \noindent (1) If $j$, $l \ne i$, then $\mj_{jl}^r$ vanishes.

\medskip \noindent (2) Suppose that $j = i$ and $l \ne i$.
Since
\[{\pa \over \pa \nu_x}\({\pa G \over \pa x_k}(x,x_i)\) = (n-2)(n-1)\gamma_n {(x-x_i)_k \over r^{n+1}} - {(x-x_i) \over r} \cdot \nabla_x \({\pa H(x,x_i) \over \pa x_k}\)\]
on $\pa B^n(x_i,r)$ and
\[G(x,x_l) = G(x_i,x_l) + (x-x_i) \cdot \nabla_x G(x_i,x_l) + O\(|x-x_i|^2\)\]
near the point $x_i$, we discover
\[\mj_{il}^r = \int_{\pa B^n(x_i,r)} \left[{\pa \over \pa \nu_x}\({\pa G \over \pa x_k}(x,x_i)\)G(x,x_l) - {\pa G \over \pa x_k}(x,x_i) {\pa G \over \pa \nu_x}(x,x_l) \right] = {\pa G \over \pa x_k}(x_i,x_l).\]

\medskip \noindent (3) In the case that $j \ne i$ and $l = i$, a similar argument in (2) applies, yielding
\[\mj_{ji}^r = {\pa G \over \pa x_k}(x_i,x_j).\]

\medskip \noindent (4) Assume that $j = l = i$. Then Green's identity \eqref{eq-green-id} and Lemma \ref{lem-Green} show that
\begin{multline*}
\mj_{ii}^r = \int_{\pa B^n(x_i,r)} \left[{\pa \over \pa \nu_x}\({\pa G \over \pa x_k}(x,x_i)\)G(x,x_i) - {\pa G \over \pa x_k}(x,x_i) {\pa G \over \pa \nu_x}(x,x_i)\right] dS \\
= - \int_{\pa \Omega} {\pa G \over \pa x_k}(x,x_i) {\pa G \over \pa \nu_x} (x,x_i) dS = -\int_{\pa \Omega} \({\pa G \over \pa \nu_x}(x,x_i)\)^2 \nu_k(x) dS = -{\pa \tau \over \pa x_k}(x_i).
\end{multline*}
We can deal with \eqref{eq-r-3} in a similar manner, which we left to the reader.
\end{proof}
\begin{proof}[Proof of Proposition \ref{prop-eigen-cha}]
We reconsider \eqref{eq-apriori-c}, but in this time we allow to put any $i \in \{1, \cdots, m\}$ and $x_k$ $(k \in \{1, \cdots, n\})$ in the place of $i_0$ and $x_1$, respectively.
By multiplying $\ep^{-{1 \over 2}-{n-1 \over n-2}}$ on both sides, we obtain
\begin{multline}\label{eq-eigen-cha-1}
\int_{\pa B^n(x_{i \ep},r)} \left[ \frac{\pa}{\pa \nu} \left\{\frac{\pa \(\ep^{-{1 \over 2}}u_{\ep}\)}{\pa x_k}\right\} \cdot \(\ep^{-{n-1 \over n-2}} v_{\ell\ep}\)
- \frac{\pa \(\ep^{-{1 \over 2}}u_{\ep}\)}{\pa x_k} \cdot \frac{\pa \(\ep^{-{n-1 \over n-2}} v_{\ell\ep}\)}{\pa \nu}\right] dS \\
= (p-\ep) \({\mu_{\ell \ep} - 1 \over \ep^{n \over n-2}}\) \cdot \left[ \ep^{-{(n-4) \over 2(n-2)}} \int_{B^n(x_{i \ep},r)} u_{\ep}^{p-1-\ep} \frac{\pa u_{\ep}}{\pa x_k} v_{\ell\ep} \right].
\end{multline}
The right-hand side of \eqref{eq-eigen-cha-1} can be computed as in \eqref{eq-apriori-d}, which turns out to be
\[\({\mu_{\ell \ep} - 1 \over \ep^{n \over n-2}}\) \left[-\lambda_i^{n-4 \over 2} d_{\ell, i, k} p C_5 + o(1)\right].\]
Meanwhile, if we let $\bl \in \mr^m$ be a nonzero column vector
\[\bl = \(\lambda_{10}^{n-2 \over 2}, \cdots, \lambda_{m0}^{n-2 \over 2}\)^T,\]
then \eqref{eq-u-ep-asym} in Lemma \ref{lem-u-asym} can be written in a vectorial form as $\ep^{-1/2}u_{\ep}(x) \to C_2 \mg(x) \bl$ (see \eqref{eq-mtx-G}).
Hence, with the aid of Proposition \ref{prop-opt-conv} and Lemma \ref{lem-jjl}, it is possible to take $\ep \to 0$ in the left-hand side of \eqref{eq-eigen-cha-1} to derive
\begin{align*}
&\ C_1C_2 \bl^T \left[\int_{\pa B^n(x_i,r)} \left\{\({\pa \over \pa \nu}{\pa \mg \over \pa x_k}(x)\)^T \mg(x) - \({\pa \mg \over \pa x_k}(x)\)^T\({\pa \mg \over \pa \nu}(x)\) \right\} dx \cdot \mm_1^{-1}\map \right. \\
&\hspace{130pt} \left. + \int_{\pa B^n(x_i,r)} \left\{\({\pa \over \pa \nu}{\pa \mg \over \pa x_k}(x)\)^T \widetilde{\mg}(x) - \({\pa \mg \over \pa x_k}(x)\)^T\({\pa \widetilde{\mg} \over \pa \nu}(x)\) \right\} dx \right] \mathbf{d}_{\ell} \\
&= C_1C_2 \bl^T \left[\mj_{ik}\mm_1^{-1}\map + \overline{\mk}_{ik} \right] \mathbf{d}_{\ell}
\end{align*}
where $\mj_{ik}$ is an $m \times m$ matrix having $\mj_{jl;ik}^r$ defined in \eqref{eq-r-2} as its components, namely, $\mj_{ik} = \(\mj_{jl;ik}^r\)_{1 \le j,l \le m}$ for each fixed $i,\ k \in \{1, \cdots, m\}$,
and $\overline{\mk}_{ik} = \(\overline{\mk}_{jb;ik}\)_{1 \le j \le m, 1 \le b \le mn}$ is an $m \times mn$ matrix whose components are
\[\overline{\mk}_{j, (l-1)n+q;ik} = \lambda_l^{n \over 2} \mk_{jl;ikq}^r = \begin{cases}
0 &\text{if } j \ne i \text{ and } l \ne i,\\
\lambda_l^{n \over 2} \dfrac{\pa^2 G}{\pa x_k \pa y_q} (x_i,x_l) &\text{if } j = i \text{ and } l \ne i,\\
\lambda_i^{n \over 2} \dfrac{\pa^2 G}{\pa x_k \pa x_q} (x_i,x_j) &\text{if } j \ne i \text{ and } l = i,\\
-\lambda_i^{n \over 2} \dfrac{1}{2} \dfrac{\pa^2 \tau}{\pa x_k \pa x_q} (x_i) &\text{if } j = l = i,
\end{cases}\]
for $j,\ l,\ i \in \{1, \cdots, m\}$ and $q,\ k \in \{1, \cdots, n\}$.
From direct computations especially using that
\[\lambda_i\(\boldsymbol\lambda^T \mj_{ik}\)_j = \begin{cases}
\lambda_i^{n \over 2} \dfrac{\pa G}{\pa x_k}(x_i, x_j) &\text{if } i \ne j,\\
\lambda_i \sum\limits_{l \ne i} \lambda_l^{n-2 \over 2} \dfrac{\pa G}{\pa x_k}(x_i,x_j) - \lambda_i^{n \over 2} \dfrac{\pa \tau}{\pa x_k}(x_i) = - \lambda_i^{n \over 2} \dfrac{1}{2} \dfrac{\pa \tau}{\pa x_k}(x_i) &\text{if } i = j,
\end{cases}\]
for $\boldsymbol\lambda^T \mj_{ik} = \(\(\boldsymbol\lambda^T \mj_{ik}\)_1, \cdots, \(\boldsymbol\lambda^T \mj_{ik}\)_m\) \in \mr^m$, we conclude
\[\ma_2 \mathbf{d}_{\ell} = \left[ \map^T\mm_1^{-1}\map + \mq \right] \mathbf{d}_{\ell} = \(-{pC_5 \over C_1C_2}\) \lim_{\ep \to 0} \({\mu_{\ell \ep} - 1 \over \ep^{n \over n-2}}\) \mathbf{d}_{\ell} = \rho^2_{\ell}\mathbf{d}_{\ell}\]
with matrices $\mm_1$, $\map$ and $\mq$ given in \eqref{eq-mtx-M_1}, \eqref{eq-mtx-P} and \eqref{eq-mtx-Q}.
The claim that $\mathbf{d}_{\ell_1}^T \cdot \mathbf{d}_{\ell_2}^T = 0$ can be proved as in the proof of Theorem \ref{thm-eigen-zero}, or particularly, \eqref{eq-eigen-zero-4}.
The proof is done.
\end{proof}

\section{Estimates for the $\ell$-th eigenvalues and eigenfunctions, $(n+1)(m+1) \le \ell \le (n+2)m$}\label{sec-7}
We now establish Theorem \ref{thm-eigen-second} by obtaining a series of lemmas.
In the first lemma we will compute the limit of the $\ell$-th eigenvalues as $\ep \to 0$ when $(n+1)(m+1) \le \ell \le (n+2)m$.
\begin{lem}\label{lem-apriori}
If $(n+1)(m+1) \le \ell \le (n+2)m$, we have
\[\lim_{\ep \to 0} \mu_{\ell \ep} = 1.\]
\end{lem}
\begin{proof}
By virtue of Corollary \ref{cor-apriori} or Corollary \ref{cor-apriori-2}, it is enough to show that $\limsup_{\ep \to 0} \mu_{\ell\ep} \le 1$.
Referring to \eqref{eq-u-psi}, we let $\mathcal{V}$ be a vector space whose basis is
\[\{u_{\ep,i}: 1 \le i \le m\} \cup \{\psi_{\ep, i, k}: 1 \le i \le m,\ 1 \le k \le n+1\}.\]
If we write $f \in \mathcal{V} \setminus \{0\}$ as
\[f = \sum_{i=1}^m f_i \quad \text{with} \quad f_i = a_{i0} u_{\ep, i} + \sum_{k=1}^{n+1} a_{ik} \psi_{\ep,i,k}\]
for some $(a_{10}, \cdots, a_{1(n+1)}, \cdots, a_{m0}, \cdots, a_{m(n+1)}) \in \mr^{m(n+1)} \setminus \{0\}$, then we have
\begin{align*}
\mu_{((n+2)m) \ep} &= \min_{\substack{\mathcal{W} \subset H_0^1(\Omega),\\ \text{dim}\mathcal{W} = (n+2)m}} \max_{f \in \mathcal{W} \setminus\{0\}}
\frac{\int_{\Omega}|\nabla f|^2}{(p-\ep)\int_{\Omega} f^2 u_{\ep}^{p-1-\ep}}
\le \max_{f \in \mathcal{V} \setminus\{0\}} \frac{\int_{\Omega}|\nabla f|^2}{(p-\ep)\int_{\Omega} f^2 u_{\ep}^{p-1-\ep}} \\
&\le \max_{f \in \mathcal{V} \setminus\{0\}} \max_{1 \le i \le m} \frac{\int_{\Omega} |\nabla f_i|^2}{(p-\ep)\int_{\Omega} f_i^2 u_{\ep}^{p-1-\ep}} := \max_{f \in \mathcal{V} \setminus\{0\}} \max_{1 \le i \le m} \mathfrak{a}_i,
\end{align*}
so it is sufficient to check that $\mathfrak{a}_i \le 1 + o(1)$.
If we denote $\mathfrak{a} = \mathfrak{a}_i$ for a fixed $i$ and modify the definition of $z_{\ep}$ in the proof of Proposition \ref{prop-apriori}
into $z_{\ep} = \sum_{k=1}^na_k{\pa u_{\ep} \over \pa x_k} + a_{n+1} w_{i\ep}$,
then we again have $\mathfrak{a} = 1 + \mathfrak{b}/\mathfrak{c}$.
(The definition of $\mathfrak{b}$, $\mathfrak{c}$ and $w_{i\ep}$ can be found in \eqref{eq-b}, \eqref{eq-c} and \eqref{eq-w_e}.)
Moreover computing each of the term of $\mathfrak{b}$ and $\mathfrak{c}$ as we did in the proof of Proposition \ref{prop-apriori}, we find
\[\mathfrak{b} \le C\(|\bar{a}|^2+a_{n+1}^2\)\ep \quad \text{and} \quad \mathfrak{c} \ge C \ep^{-{2 \over n-2}}|\bar{a}|^2 + Ca_{n+1}^2 \ge C\(|\bar{a}|^2 + a_{n+1}^2\),\]
from which one can conclude that $\mu_{((n+2)m)\ep} \le 1+O(\ep)$.
For more detailed computations, we ask for the reader to check the proof of Theorem 1.4 in \cite{GP}.
\end{proof}

The following lemma is the counterpart of Proposition \ref{prop-weak-conv} for $(n+1)(m+1) \le \ell \le (n+2)m$.
\begin{lem}\label{lem-weak-conv-2}
Let $(n+1)(m+1) \le \ell \le (n+2)m$.
For each $i \in \{1,\cdots,m\}$ and $d_{\ell,i,n+1} \in \mr$, converges to
\[\tilde{v}_{\ell i \ep} \rightharpoonup d_{\ell,i,n+1} \(\frac{\pa U_{1,0}}{\pa \lambda}\) \quad \text{weakly in } H^1(\mr^n).\]
\end{lem}
\begin{proof}
Lemma \ref{lem-weak-conv} (1) holds in this case also by Lemma \ref{lem-apriori}.
Therefore it is enough to show that the vector $\mathbf{d}_{\ell}$ in \eqref{eq-eigenv-first-2} is zero.

As in \eqref{eq-eigen-zero-4}, the orthogonality of $v_{\ell \ep}$ and $v_{\ell_1 \ep}$ for $m+1 \le \ell_1 \le (n+1)m$ implies $\mathbf{d}_{\ell}^T \cdot \mathbf{d}_{\ell_1}^T = 0$.
However, we also know from Proposition \ref{prop-eigen-cha} that $\{\mathbf{d}_{m+1}, \cdots, \mathbf{d}_{(n+1)m}\}$ serves a basis for $\mr^{mn}$.
Hence $\mathbf{d}_{\ell} = 0$, concluding the proof.
\end{proof}

As a consequence, we reach at
\begin{prop}
Let $\ma_3$ be the matrix \eqref{eq-mtx-A_3}.
For $(n+1)(m+1) \le \ell \le (n+2)m$, if $\rho^3_{\ell}$ is the $(\ell-(m+1)n)$-th eigenvalue of $\ma_3$, then it is positive
and the $\ell$-th eigenvalue $\mu_{\ell \ep}$ to problem \eqref{eq-lin} is estimated as
\begin{equation}\label{eq-eigen-cha-a}
\mu_{\ell \ep} = 1 + c_1\rho_\ell^3\ep + o(\ep) \quad \text{where } c_1 = {(n-2)^2C_2C_3 \over 2(n+2)C_4}.
\end{equation}
Furthermore, the nonzero vector $\hat{\mathbf{d}}_{\ell}$ in \eqref{eq-hde} is a corresponding eigenvector to $\rho^3_{\ell}$
and $\hat{\mathbf{d}}_{\ell_1}^T \cdot \hat{\mathbf{d}}_{\ell_2}^T = 0$ if $(n+1)(m+1) \le \ell_1 \ne \ell_2 \le (n+2)m$.
\end{prop}
\begin{proof}
Denote $d_{\ell, i} = d_{\ell, i, n+1}$ in the previous lemma.
Then we can recover \eqref{eq-v-conv} from Lemma \ref{lem-apriori}.
Hence the arguments in the proof of Proposition \ref{prop-weak-conv} works, giving \eqref{eq-weak-conv-3} and \eqref{eq-weak-conv-4} to us again.
From them, we conclude that $\rho^3_{\ell}$ is positive, $\hat{\mathbf{d}}_{\ell}$ is an eigenvector corresponding to $\rho_{\ell}^3$ and \eqref{eq-eigen-cha-a} is valid.
The last orthogonality assertion is deduced in the same way as one in Theorem \ref{thm-eigen-zero}. See \eqref{eq-eigen-zero-4}.
\end{proof}

\appendix
\section{An moving sphere argument}\label{sec-mov-sphere}
In this appendix, we show the following proposition by employing the moving sphere argument given in \cite{LZ} (refer also to \cite{ChaC}).
Note that it implies Proposition \ref{prop-LZ} at once.
\begin{prop}\label{prop-moving}
Let $r_0 > 0$ be fixed and $p = (n+2)/(n-2)$ as above.
Suppose that a family $\{u_{\ep}\}_{\ep}$ of positive $C^2$-functions which satisfy
\[-\Delta u_{\ep} = u_{\ep}^{p-\ep} \quad \text{in } B^n\(0, \ep^{-\alpha_0}r_0\), \quad
\|u_{\ep}\|_{L^{\infty}\(B^n\(0, \ep^{-\alpha_0}r_0\)\)} \le c\]
for some $c > 0$, and
\begin{equation}\label{eq-a-eq2}
\lim_{\ep \to 0} u_{\ep} (x) = U_{1,0}(x)\quad \text{weakly in } H^1(\mr^n).
\end{equation}
Then there are constants $C > 0$ and $0 < \delta_0 < r_0$ independent of $\ep > 0$ such that
\[u_{\ep} (x) \le CU_{1,0}(x) \quad \text{for all } x \in B^n\(0, \ep^{-\alpha_0}\delta_0\).\]
\end{prop}

\noindent Before conducting its proof, we introduce Green's function $G_R$ of $-\Delta$ in $B^n(0,R)$ for each $R > 0$ with zero Dirichlet boundary condition.
By the scaling invariance, we have
\[G_R (x,y) = G_1 \(\frac{x}{R}, \frac{y}{R}\)\frac{1}{R^{n-2}} \quad \text{for } x, y \in B^n(0,R).\]
Thus we can decompose Green's function in $B^n(0,R)$ into its singular part and regular part as follows:
\begin{equation}\label{eq-G-R}
G_R (x,y) = {\gamma_n \over |x-y|^{n-2}} - {1 \over R^{n-2}}H_1 \({x \over R},{y \over R}\) \quad \text{for } x, y \in B^n(0,R).
\end{equation}
See \eqref{eq-H} for the definition of the normalizing constant $\gamma_n$.

\medskip
Now we begin to prove Proposition \ref{prop-moving}.
By \eqref{eq-a-eq2} and elliptic regularity, for arbitrarily given $\zeta_1 > 0$ and any compact set $K \subset \mr^n$, there is $\ep_1 > 0$ such that it holds
\begin{equation}\label{eq-a-con-com}
\|u_{\ep} - U_{1,0}\|_{C^2(K)} \le \zeta_1 \quad \text{for } \ep \in (0, \ep_1).
\end{equation}
Let us define the Kelvin transform of $u_{\ep}$:
\begin{equation}\label{eq-kelvin}
u_{\ep}^{\lambda}(x) = \(\frac{\lambda}{|x|}\)^{n-2} u_{\ep}\(x^{\lambda}\),\quad x^{\lambda} = \frac{\lambda^2 x}{|x|^2}
\quad \text{for } |x^{\lambda}| < \ep^{-\alpha_0}r_0
\end{equation}
and the difference $w_{\ep}^{\lambda} = u_{\ep} - u_{\ep}^{\lambda}$ between $u_{\ep}$ and it.
Then we have
\begin{equation}\label{eq-a-w}
-\Delta w_{\ep}^{\lambda} = u_{\ep}^{p-\ep} - \({\lambda \over |x|}\)^{(n-2)\ep} \(u_{\ep}^{\lambda}\)^{p-\ep} \ge u_{\ep}^{p-\ep} - \(u_{\ep}^{\lambda}\)^{p-\ep} = \xi_{\ep}(x) w_{\ep}^{\lambda} \quad \text{for } |x| \ge \lambda
\end{equation}
where
\[\xi_{\ep} (x) = \begin{cases}
\dfrac{u_{\ep}^{p-\ep}-\(u_{\ep}^{\lambda}\)^{p-\ep}}{u_{\ep} -u_{\ep}^{\lambda}}(x) &\text{if } u_{\ep}(x) \ne u_{\ep}^{\lambda}(x),\\
(p-\ep)u_{\ep}^{p-1-\ep}(x) &\text{if } u_{\ep}(x) = u_{\ep}^{\lambda}(x).
\end{cases}\]

\begin{lem}\label{lem-moser-1}
For any $\zeta_2 > 0$, there exist small constants $\delta_1 > 0$ and $\ep_2 > 0$ such that
\begin{equation}\label{eq-inf-bound}
\min_{|y|=r} u_{\ep} (y) \le (1+ \zeta_2) U_{1,0}(r) \quad \text{for } 0 < r := |x| \le \ep^{-\alpha_0}\delta_1 \text{ and any } \ep \in (0, \ep_2).
\end{equation}
\end{lem}
\begin{proof}
We first choose a candidate $\delta_1 \in (0, r_0)$ for which \eqref{eq-inf-bound} will have the validity.
Fix a sufficiently small value $\eta_1 > 0$ and a number $R_0 > 0$ such that it holds
\begin{equation}\label{eq-u_k-3}
u_{\ep}^{\lambda}(x) \le \(1+{\zeta_2 \over 4}\) \beta_n |x|^{2-n} \quad \text{for any } 0 < \lambda \le 1+\eta_1 \text{ and } |x| \ge R_0
\end{equation}
provided $\ep > 0$ small enough, where $\beta_n = (n(n-2))^{p-1}$ is the constant appeared in \eqref{eq-U}.
Take $\lambda_1 = 1-\eta_1$ and $\lambda_2 = 1+\eta_1$.
If $\lambda = \lambda_1$, because $U_{1,0}^{\lambda} = U_{\lambda^2,0}$ for any $\lambda > 0$ and $u_{\ep} \to U_{1,0}$ in $C^1$-uniformly over compact subsets of $\mr^n$ as $\ep \to 0$, by enlarging $R_0 > 0$ if necessary,
we can find a number $\eta_2 > 0$ small such that
\begin{equation}\label{eq-w_k}
w_{\ep}^{\lambda_1}(x) > 0 \quad \text{for } \lambda_1 < |x| \le R_0, \quad u_{\ep}^{\lambda_1}(x) \le (1-2\eta_2)\beta_n|x|^{2-n} \quad \text{for } |x| \ge R_0
\end{equation}
and
\begin{equation}\label{eq-u_k}
\int_{B^n(0,R_0)} u_{\ep}^{p-\ep}(x)dx \ge \(1-{\eta_2 \over 2}\) \int_{\mr^n} U_{1,0}^p(x)dx
\end{equation}
for sufficiently small $\ep > 0$.
On the other hand, provided $\delta_1 > 0$ small enough, the inequality
\begin{equation}\label{eq-u_k-2}
u_{\ep} (x) \ge (1-\eta_2) \beta_n|x|^{2-n} \quad \text{for } R_0 \le |x| \le \ep^{-\alpha_0} \delta_1
\end{equation}
can be reasoned in the following way. If we choose a function $\hat{u}_{\ep}$ which solves
\[-\Delta \hat{u}_{\ep} = u_{\ep}^{p-\ep} \quad \text{in } B^n\(0, \ep^{-\alpha_0}\) \quad \text{and} \quad \hat{u}_{\ep} = 0 \quad \text{on } \left\{|x| = \ep^{-\alpha_0} \right\},\]
then the comparison principle tells us that $u_{\ep} \ge \hat{u}_{\ep}$.
Since Green's function is always positive, we can make
\[H_1\(\ep^{-\alpha_0}x, \ep^{-\alpha_0}y\) \le {\eta_2 \gamma_n \over 4} \cdot
{\ep^{-\alpha_0(n-2)} \over |x-y|^{n-2}} \quad \text{for } x,y \in B^n\(0, \ep^{-\alpha_0} \delta_1\)\]
by taking $\delta_1$ small, and the relation $|x-y| \le (1-1/l)|x|$ holds for $|x| \ge lR_0$ and $|y| \le R_0$ given any $l \in (1,\infty)$,
we see from \eqref{eq-G-R} and \eqref{eq-u_k} that
\begin{align*}
\hat{u}_{\ep}(x) &= \int_{B^n\(0, \ep^{-\alpha_0}\)} u_{\ep}^{p-\ep}(y) G_{\ep^{-\alpha_0}} (x,y) dy
\ge \(1-{\eta_2 \over 4}\) \int_{B^n\(0, \ep^{-\alpha_0}\delta_1 \)} u_{\ep}^{p-\ep}(y){\gamma_n \over |x-y|^{n-2}}dy \\
&\ge \(1-{\eta_2 \over 2}\) \(\int_{B^n(0,R_0)} u_{\ep}^{p-\ep}(y) dy\) {\gamma_n \over |x|^{n-2}} \ge \(1-\eta_2\) \(\int_{\mr^n} U_{1,0}^p(y)dy\) {\gamma_n \over |x|^{n-2}}\\
&= \(1-\eta_2\) {\beta_n \over |x|^{n-2}} \qquad \text{for } lR_0 \le |x| \le \ep^{-\alpha_0}\delta_1
\end{align*}
by choosing $l$ large enough.
Also if $|x| \le lR_0$, the uniform convergence of $u_{\ep}$ to $U_{1,0}$ implies $u_{\ep} (x) \ge (1-\eta_2) \beta_n|x|^{2-n}$ for $\ep > 0$ sufficiently small.
This shows the validity of \eqref{eq-u_k-2}.

Fixing $\delta_1 > 0$ for which \eqref{eq-u_k-2} is valid, suppose that \eqref{eq-inf-bound} does not hold on the contrary.
Then there are sequences $\{\ep_k\}_{k=1}^{\infty}$ and $\{r_k\}_{k=1}^{\infty}$ such that $\ep_k \to 0$, $r_k \in \(0, \ep^{-\alpha_0} \delta_1\)$ and
\[\min_{|x|=r_k} u_{\ep_k} (x) > (1+\zeta_2) U_{1,0}(r_k).\]
Set $u_k = u_{\ep_k}$ for brevity.
Since $u_k \to U_{1,0}$ uniformly on any compact set, it should hold that $r_k \to \infty$. Therefore
\begin{equation}\label{eq-contrary}
\min_{|x|=r_k} u_k (x) \ge \(1+{\zeta_2 \over 2}\) \beta_n r_k^{2-n}.
\end{equation}

To deduce a contradiction, let us apply the moving sphere method to $w_k^{\lambda} = u_k - u_k^{\lambda}$ for the parameters $\lambda_1 \le \lambda \le \lambda_2$.
Define $\bar{\lambda}_k$ by
\[\bar{\lambda}_k = \sup \left\{ \lambda \in [\lambda_1, \lambda_2] : w_k^{\mu} \ge 0 \text{ in } \Sigma_{\mu} \text{ for all } \lambda_1 \le \mu \le \lambda \right\} \text{ where } \Sigma_{\mu} = \{x \in \mr^n: \mu < |x| < r_k\}.\]
We claim that $\bar{\lambda}_k = \lambda_2$ for sufficiently large $k \in \mathbb{N}$.
First of all, putting together with \eqref{eq-w_k} and \eqref{eq-u_k-2}, we discover that $w_k^{\lambda_1} > 0$ in $\Sigma_{\lambda_1}$, so $\bar{\lambda}_k \ge \lambda_1$.
Recall from \eqref{eq-a-w} that
\[-\Delta w_k^{\bar{\lambda}_k} + (\xi_{\ep_k})_-w_k^{\bar{\lambda}_k} \ge (\xi_{\ep_k})_+ w_k^{\bar{\lambda}_k} \ge 0 \quad \text{in } \Sigma_{\bar{\lambda}_k}.\]
Moreover, from \eqref{eq-contrary} and \eqref{eq-u_k-3} we have $w_k^{\bar{\lambda}_k} >0$ on $\pa B^n(0, r_k)$.
Thus by the maximum principle and Hopf's lemma we have
\[w_k^{\bar{\lambda}_k} > 0 \quad \text{in } \Sigma_{\bar{\lambda}_k} \quad \text{and} \quad
\frac{\pa w_k^{\bar{\lambda}_k}}{\pa \nu} < 0 \quad \text{on } \pa B^n\(0, \bar{\lambda}_k\)\]
where $\nu$ is the unit outward normal vector.
However this means that if $\bar{\lambda}_k < \lambda_2$, then $w_k^{\mu} \ge 0$ in $\Sigma_{\mu}$ even after taking a slightly larger value of $\mu$ than $\bar{\lambda}_k$,
which contradicts the maximality of $\bar{\lambda}_k$.
Hence our claim is justified.
Consequently, taking a limit $k \to \infty$ to $w_k^{\lambda_2} \ge 0$ in $\Sigma_{\lambda_2}$ allows one to get
\[U_{1,0}(x) \ge U_{1,0}^{\lambda_2} (x) \quad \text{in } |x| \ge \lambda_2,\]
but it cannot be possible since $\lambda_2 > 1$.
Thus \eqref{eq-inf-bound} should be true.
\end{proof}

The following lemma completes our proof of Proposition \ref{prop-moving}.
\begin{lem}\label{lem-moser-2}
For some constant $C>0$ and parameter $\delta_0 \in (0,\delta_1)$, we have
\[u_{\ep} (x) \le C U_{1,0}(x) \quad \text{for } |x| \le \ep^{-\alpha_0} \delta_0\]
provided that $\ep > 0$ is sufficiently small.
Here $\delta_1 > 0$ is the number chosen in the proof of the previous Lemma.
\end{lem}
\begin{proof}
Argue as in the proof of Lemma 2.4 in \cite{LZ} employing Lemma \ref{lem-moser-1} above.
In that paper, the statement of the lemma as well as its proof are written for a sequence $\{u_{\ep_k}\}_{k=1}^{\infty}$ of solutions,
but they apply to a family $\{u_{\ep}\}_{\ep}$ as well.
To proceed our proof, we substitute $G_k$, $R_k$ and $v_k$ in \cite{LZ} with Dirichlet Green's function $G_{\ep^{-\alpha_0}\delta_1}$ of $-\Delta$ in $B^n(0,\ep^{-\alpha_0}\delta_1)$, $R_{\ep} = \ep^{-\alpha_0} \delta_1\delta_2$ and $u_{\ep}$ where $\delta_2 \in (0,1)$ is a sufficiently small number.
\end{proof}

\end{document}